\documentclass[english]{article}
\usepackage[T1]{fontenc}
\usepackage[latin9]{inputenc}
\usepackage{color}
\usepackage{float}
\usepackage{textcomp}
\usepackage{amsmath}
\usepackage{amsthm}
\usepackage{amssymb}
\usepackage{graphicx}

\makeatletter

\providecommand{\tabularnewline}{\\}
\floatstyle{ruled}
\newfloat{algorithm}{tbp}{loa}
\providecommand{\algorithmname}{Algorithm}
\floatname{algorithm}{\protect\algorithmname}

\newcommand{\lyxaddress}[1]{
	\par {\raggedright #1
	\vspace{1.4em}
	\noindent\par}
}
\theoremstyle{plain}
\newtheorem{thm}{\protect\theoremname}
\theoremstyle{definition}
\newtheorem{defn}[thm]{\protect\definitionname}
\theoremstyle{plain}
\newtheorem{lem}[thm]{\protect\lemmaname}
\theoremstyle{remark}
\newtheorem{rem}[thm]{\protect\remarkname}
\theoremstyle{plain}
\newtheorem{prop}[thm]{\protect\propositionname}
\newenvironment{lyxlist}[1]
	{\begin{list}{}
		{\settowidth{\labelwidth}{#1}
		 \setlength{\leftmargin}{\labelwidth}
		 \addtolength{\leftmargin}{\labelsep}
		 }}
	{\end{list}}
\theoremstyle{remark}
\newtheorem*{acknowledgement*}{\protect\acknowledgementname}

\usepackage{babel}

\usepackage{babel}
\providecommand{\definitionname}{Definition}
\providecommand{\lemmaname}{Lemma}
\providecommand{\propositionname}{Proposition}
\providecommand{\remarkname}{Remark}
\providecommand{\theoremname}{Theorem}

\makeatother

\usepackage{babel}
\providecommand{\acknowledgementname}{Acknowledgement}
\providecommand{\definitionname}{Definition}
\providecommand{\lemmaname}{Lemma}
\providecommand{\propositionname}{Proposition}
\providecommand{\remarkname}{Remark}
\providecommand{\theoremname}{Theorem}

\begin{document}
\title{A stochastic optimization algorithm for analyzing planar central and
balanced configurations in the $n$-body problem }
\author{Alexandru Doicu$^{1}$\thanks{E-mail address: alexandru.doicu@math.uni-augsburg.de},
Lei Zhao$^{1}$\thanks{E-mail address: lei.zhao@math.uni-augsburg.de}
and Adrian Doicu$^{2}$\thanks{Corresponding author. E-mail address: adrian.doicu@dlr.de}}
\maketitle

\lyxaddress{$^{1}$Institut für Mathematik, Universität Ausgburg, Augsburg 86135,
Germany}

\lyxaddress{$^{2}$Institut für Methodik der Fernerkundung (IMF), Deutsches Zentrum
für Luft- und Raumfahrt (DLR), Oberpfaffenhofen 82234, Germany }
\begin{abstract}
A stochastic optimization algorithm for analyzing planar central and
balanced configurations in the $n$-body problem is presented. We
find a comprehensive list of equal mass central configurations satisfying
the Morse equality up to $n=12$. We show some exemplary balanced
configurations in the case $n=5$, as well as some balanced configurations
without any axis of symmetry in the cases $n=4$ and $n=10$. 
\end{abstract}

\section{Introduction}

The $n$-body problem is the problem of predicting motions of a group
of celestial objects interacting with each other gravitationally.
A central configuration is an initial configuration such that if the
particles were all released with zero velocity, they would all collapse
toward the center of mass at the same time. In the planar case, central
configurations serve as initial positions for periodic solutions which
preserve the shape of the configuration. More generally, a balanced
configuration leads to (periodic or quasi-periodic) relative equilibria
in higher dimensional Euclidean spaces.

Several fundamental studies have addressed the questions of existence,
finiteness, and classification of central configurations. In this
context, it should be pointed out that the finiteness problem was
included by Smale (1998) in his list of problems for this century.
In the case $n=3$, all central configurations are known to Euler
(1767) and Lagrange (1772). In chronological order, we mention the
following non-exhaustive list of results concerning enumeration of
central configurations:
\begin{enumerate}
\item Xia (1991) made an exact count on the number of central configurations
for the $n$-body problem with many small masses, 
\item Moeckel (2001) showed a generic finiteness of $(n-2)$-dimensional
central configurations of $n$ bodies, 
\item Hampton and Moeckel (2006) proved the finiteness for all positive
masses in the case $n=4$, 
\item Hampton and Jensen (2011) strengthened the result of Moeckel (2001)
in the case $n=5$, and
\item Albouy and Kaloshin (2012) established a generic finiteness of planar
central configurations in the case $n=5$. 
\end{enumerate}
An excellent concise survey on this topic can be found in Moeckel
(2014b). 

Aside from theoretical studies, numerical approaches for analyzing
central configurations are relevant in practice as they give an instructive
picture on this matter. The main ambition of a numerical method is
to find all (approximating) central configurations for a given number
$n$ of fixed positive masses. In this context, the following contributions
deserve to be mentioned.
\begin{enumerate}
\item Moeckel (1989) found a list of central configurations of $n$ equal
masses by using a stochastic algorithm based on the Multistart method,
i.e., by repeatedly applying the steepest descent Newton's method
with randomly chosen initial conditions.
\item Using a similar solution method (but with a root-finding routine taken
from the SLATEC library), Ferrario (2002) approximately computed all
planar central configurations with equal masses for $n\leq9$ and
found 64 central configurations in the case $n=10$. 
\item Lee and Santoprete (2009) computed all isolated\footnote{Actually all these configurations are isolated, as has been confirmed
by the study of Moczurad and Zgliczynski (2019).} central configurations of the five-body problem with equal masses.
This was accomplished by using the polyhedral homotopy method to approximate
all the isolated solutions of the Albouy-Chenciner equations. The
existence of exact solutions, in a neighborhood of the approximated
ones, was verified by using the Krawczyk operator method. 
\item Moczurad and Zgliczynski (2019) computed all planar central configurations
with equal masses for $n\leq7$. Standard interval arithmetic tools
were used in conjunction with the Krawczyk operator method to establish
the existence and local uniqueness of the solutions. As in Lee and
Santoprete (2009), they also show that there exists non-symmetric
central configurations for $n=7,8,9$. In a subsequent paper, Moczurad
and Zgliczynski (2020) presented a complete list of spatial central
configurations with equal masses for $n=5,6$, and provided their
Euclidean symmetries. 
\end{enumerate}
Adopting the solution method proposed by Moeckel (1989) and Ferrario
(2002), we develop a stochastic optimization algorithm to analyze
planar central and balanced configurations. The paper is organized
as follows. A succinct mathematical description of planar central
and balanced configurations is provided in Section 2. In Section 3,
the stochastic optimization algorithm is presented, while in Section
4, several approaches for testing the solutions are described. Numerical
results are given in Section 5, and some conclusions are summarized
in Section 6.

\section{Planar central and balanced configurations }

Consider $n$ point masses $m_{1},\ldots,m_{n}>0$ with positions
$\mathbf{q}_{1},\ldots,\mathbf{q}_{n}$, where $\mathbf{q}_{i}=(x_{i},y_{i})^{T}\in\mathbb{R}^{2}$.
The vector 
\[
\mathbf{q}=\left(\begin{array}{c}
\mathbf{q}_{1}\\
\vdots\\
\mathbf{q}_{n}
\end{array}\right)\in\mathbb{R}^{N}
\]
with $N=2n$ will be called a configuration. Let $\Delta$ be the
subspace of $\mathbb{R}^{N}$ consisting of collisions, i.e., 
\begin{align*}
\Delta & =\{\mathbf{q}=(\mathbf{q}_{1}^{T},\ldots,\mathbf{q}_{n}^{T})^{T}\mid\mathbf{q}_{i}=\mathbf{q}_{j}\text{ for some }i\not=j\}.
\end{align*}
The Newtonian force function (negative of the Newtonian potential)
$U_{n}(\mathbf{q})$ for the configuration $\mathbf{q}\in\mathbb{R}^{N}\backslash\Delta$
is defined by 
\[
U_{n}(\mathbf{q})=\sum_{1\leq i<j\leq n}\frac{m_{i}m_{j}}{||\mathbf{q}_{j}-\mathbf{q}_{i}||},
\]
where $||\cdot||$ is the Euclidean norm in $\mathbb{R}^{2}$, and
for $i\in\{1,\ldots,n\}$, we denote by $\nabla_{i}U_{n}(\mathbf{q})\in\mathbb{R}^{2}$
the derivative of $U_{n}$ with respect to the coordinates of $\mathbf{q}_{i}$,
i.e., 
\[
\nabla_{i}U_{n}(\mathbf{q})=\left(\begin{array}{c}
\dfrac{\partial U_{n}}{\partial x_{i}}(\mathbf{x})\\
\dfrac{\partial U_{n}}{\partial y_{i}}(\mathbf{x})
\end{array}\right)=\sum_{\substack{j=1\\
j\not=i
}
}^{n}\frac{m_{i}m_{j}}{||\mathbf{q}_{j}-\mathbf{q}_{i}||^{3}}(\mathbf{q}_{j}-\mathbf{q}_{i}).
\]

Let 
\[
\mathbf{c}=\Bigl(\sum_{i=1}^{n}m_{i}\Bigr)^{-1}\sum_{i=1}^{n}m_{i}\mathbf{q}_{i}
\]
be the center of mass of the configuration, and $\mathbf{S}\in\mathbb{R}^{2\times2}$
a positive definite symmetric matrix. 
\begin{defn}
\label{def:A-configuration-}A configuration $\mathbf{q}=(\mathbf{q}_{1}^{T},\ldots,\mathbf{q}_{n}^{T})^{T}\in\mathbb{R}^{N}\backslash\Delta$
is said to form a balanced configuration with respect to the matrix
$\mathbf{S}$ (in short $\text{BC}(\mathsf{\mathbf{S}})$) if there
exists a $\lambda\in\mathbb{R}\backslash\{0\}$ such that the equations
\begin{equation}
\nabla_{i}U_{n}(\mathbf{q})+m_{i}\lambda\mathbf{S}(\mathbf{q}_{i}-\mathbf{c})=\mathbf{0},\label{eq:1}
\end{equation}
are satisfied for all $i=1,\ldots,n$. A configuration $\mathbf{q}=(\mathbf{q}_{1}^{T},\ldots,\mathbf{q}_{n}^{T})^{T}\in\mathbb{R}^{N}\backslash\Delta$
is said to form a central configuration (in short CC) if there exists
a $\lambda\in\mathbb{R}\backslash\{0\}$ for which Eqs. (\ref{eq:1})
are satisfied with $\mathbf{S}=\mathbf{I}_{2\times2}$. 
\end{defn}

As such, central configurations are special cases of balanced configurations.
As described by Albouy and Chenciner (1998), balanced configurations
are those configurations which admit (in general quasi-periodic) relative
equilibrium motions of the $n$-body problem in some Euclidean space
of dimension high enough. Central configurations are those balanced
configurations for which the corresponding relative equilibrium motions
are periodic.

Consider the diagonal action of $\text{O}(2)$ on $\mathbb{R}^{N}$,
defined by 
\begin{align}
\text{O}(2)\times\mathbb{R}^{N}\backslash\Delta & \rightarrow\mathbb{R}^{N}\backslash\Delta\nonumber \\
(\mathbf{O},\mathbf{q}) & \mapsto\mathbf{O}\mathbf{q},\label{eq:1a}
\end{align}
where
\[
\mathbf{O}\mathbf{q}=\left(\begin{array}{c}
\mathbf{O}\mathbf{q}_{1}\\
\vdots\\
\mathbf{O}\mathbf{q}_{n}
\end{array}\right)
\]

\begin{lem}
\label{lem:Let--be}Let $\mathbf{q}\in\mathbb{R}^{N}\backslash\Delta$
be a \emph{$\text{BC}(\mathbf{S})$} and $\mathbf{O}\in\text{O}(2)$
an orthogonal matrix. Then $\mathbf{O}\mathbf{q}$ is a \emph{$\text{BC}(\mathbf{O}\mathbf{S}\mathbf{O}^{T})$.} 
\end{lem}

\begin{proof}
For $\mathbf{O}\in\text{O}(2)$ and every $i\in\{1,\ldots,n\}$, we
define a new configuration $\widehat{\mathbf{q}}_{i}$ by $\widehat{\mathbf{q}}_{i}=\mathbf{O}\mathbf{q}_{i}$.
Taking into account that $\mathbf{q}$ solves the system of equations
\begin{equation}
\sum_{j\not=i}\frac{m_{i}m_{j}}{||\mathbf{q}_{j}-\mathbf{q}_{i}||^{3}}(\mathbf{q}_{j}-\mathbf{q}_{i})+m_{i}\lambda\mathbf{S}(\mathbf{q}_{i}-\mathbf{c})=\boldsymbol{0},\quad i=1,\ldots,N,\label{eq:2}
\end{equation}
for some $\lambda\in\mathbb{R}\backslash\{0\}$, we multiply each
equation from the left by $\mathbf{O}$, and obtain 
\begin{align*}
\boldsymbol{0} & =\sum_{\substack{j=1\\
j\not=i
}
}^{n}\frac{m_{i}m_{j}}{||\widehat{\mathbf{q}}_{j}-\widehat{\mathbf{q}}_{i}||^{3}}(\widehat{\mathbf{q}}_{j}-\widehat{\mathbf{q}}_{i})+m_{i}\lambda\mathbf{O}\mathbf{S}(\mathbf{q}_{i}-\mathbf{c})\\
 & =\sum_{\substack{j=1\\
j\not=i
}
}^{n}\frac{m_{i}m_{j}}{||\widehat{\mathbf{q}}_{j}-\widehat{\mathbf{q}}_{i}||^{3}}(\widehat{\mathbf{q}}_{j}-\widehat{\mathbf{q}}_{i})+m_{i}\lambda\mathbf{O}\mathbf{S}\mathbf{O}^{T}(\widehat{\mathbf{q}}_{i}-\widehat{\mathbf{c}}),\quad i=1,\ldots,N,
\end{align*}
where $\widehat{\mathbf{c}}$, defined by $\widehat{\mathbf{c}}=\mathbf{O}\mathbf{c}$,
is the center of mass of $\widehat{\mathbf{q}}$. Thus, if $\mathbf{q}$
forms a $\text{BC}(\mathbf{S})$ then $\widehat{\mathbf{q}}=(\widehat{\mathbf{q}}_{1}^{T},\ldots,\widehat{\mathbf{q}}_{n}^{T})^{T}$
forms a $\text{BC}(\mathbf{O}\mathbf{S}\mathbf{O}^{T})$. 
\end{proof}
\begin{rem}
\label{rem:In-general,-note}The following comments can be made here.
\begin{enumerate}
\item If $\mathbf{q}\in\mathbb{R}^{N}\backslash\Delta$ forms a central
configuration and $\mathbf{O}\in\text{O}(2)$ is an orthogonal matrix,
then by Lemma \ref{lem:Let--be}, $\mathbf{O}\mathbf{q}$ is also
a central configuration. 
\item From Lemma \ref{lem:Let--be} we may assume that the positive definite
$2\times2$ matrix $\mathbf{S}$ is diagonal, i.e., 
\begin{equation}
\mathbf{S}=\left(\begin{array}{cc}
\sigma_{x} & 0\\
0 & \sigma_{y}
\end{array}\right)\label{eq:6}
\end{equation}
with $\sigma_{x},\sigma_{y}>0$. Indeed, if $\mathbf{q}$ is a $\text{BC}(\mathbf{S})$
and $\mathbf{O}\in\text{O}(2)$ is an orthogonal matrix such that
$\mathbf{O}\mathbf{S}\mathbf{O}^{T}=\text{diag}(\sigma_{x},\sigma_{y})$,
then $\mathbf{O}\mathbf{q}$ is a $\text{BC}(\mathbf{O}\mathbf{S}\mathbf{O}^{T})=\text{BC}(\text{diag}(\sigma_{x},\sigma_{y}))$.
It is obvious that the converse result also holds. Thus, the set of
$\text{BC}(\mathbf{S})$ in $\mathbb{R}^{N}\backslash\Delta$ corresponds
one to one with the set of $\text{BC}(\text{diag}(\sigma_{x},\sigma_{y}))$.
This bijection is given by the action defined in (\ref{eq:1a}). 
\end{enumerate}
\end{rem}

\begin{rem}
\label{rem:Lemma--treats}Lemma \ref{lem:Let--be} only treats rotations
and reflexions of balanced configurations. It is also possible to
dilate a balanced configuration. Let $\mathbf{q}$ be a configuration
that forms a $\text{BC}(\mathbf{S})$ with respect to some $\lambda\in\mathbb{R}\backslash\{0\}$
and let $\varsigma>0$ be some positive real number. Then $\varsigma\mathbf{q}=(\varsigma\mathbf{q}_{1}^{T},\ldots,\varsigma\mathbf{q}_{n}^{T})^{T}$
also forms a $\text{BC}(\mathbf{S})$ with respect to $\varsigma^{-3}\lambda$. 
\end{rem}

Thanks to Remark \ref{rem:In-general,-note} we may assume that the
matrix $\mathbf{S}$ is diagonal, i.e., $\mathbf{S}$ is as in Eq.
(\ref{eq:6}) with $\sigma_{x},\sigma_{y}>0$. In the following, we
use the results established by Moeckel (2014a) to introduce a Morse
theoretical approach to Eq. (\ref{eq:1}) and a notion of non-degeneracy
for balanced configurations.

For $\boldsymbol{\xi},\boldsymbol{\eta}\in\mathbb{R}^{2}$, we define
their inner product with respect to the positive definite diagonal
matrix $\mathbf{S}$ by 
\[
\bigl\langle\boldsymbol{\xi},\boldsymbol{\eta}\bigr\rangle_{\mathbf{S}}:=\boldsymbol{\xi}^{T}\mathbf{S}\boldsymbol{\eta}\quad||\boldsymbol{\xi}||_{\mathbf{S}}^{2}=\boldsymbol{\xi}^{T}\mathbf{S}\mathbf{\boldsymbol{\xi}},
\]
and accordingly, the \emph{$\mathbf{S}$-}weighted moment of inertia
by 
\[
I_{\mathbf{S}}(\mathbf{q})=\sum_{j=1}^{n}m_{j}(\mathbf{q}_{j}-\mathbf{c})^{T}\mathbf{S}(\mathbf{q}_{j}-\mathbf{c})=\sum_{j=1}^{n}m_{j}||\mathbf{q}_{j}-\mathbf{c}||_{\mathbf{S}}^{2}.
\]

\begin{rem}
Assume that the configuration $\mathbf{q}$ is a $\text{BC}(\mathbf{S})$.
By taking the inner product of Eq. (\ref{eq:1}) with $\mathbf{q}_{i}-\mathbf{c}$
and summing up over all $i=1,\ldots,n$ , we obtain 
\[
\sum_{i=1}^{n}(\mathbf{q}_{i}-\mathbf{c})^{T}\nabla_{i}U_{n}(\mathbf{q})+\lambda I_{\mathbf{S}}(\mathbf{q})=0.
\]
Thus, 
\[
\lambda=-\frac{1}{I_{\mathbf{S}}(\mathbf{q})}\sum_{i=1}^{n}(\mathbf{q}_{i}-\mathbf{c})^{T}\nabla_{i}U_{n}(\mathbf{q})=\frac{U_{n}(\mathbf{q})}{I_{\mathbf{S}}(\mathbf{q})}>0.
\]
The last equality follows from the translation invariance of $U_{n}$
and Euler's homogeneous function theorem. From this result it is apparent
that in the definition of balanced configurations, the parameter $\lambda$
cannot be chosen arbitrary; it depends on $\mathbf{q}$ and $\mathbf{S}$. 
\end{rem}

We define the \emph{$\mathbf{S}$-}normalized configuration space
$\mathcal{N}(\mathbf{S})$ as 
\[
\mathcal{N}(\mathbf{S})=\{\mathbf{q}\in\mathbb{R}^{N}\backslash\Delta\mid\mathbf{c}(\mathbf{q})=0,\,I_{\mathbf{S}}(\mathbf{q})=1\}\subset\mathbb{R}^{N}.
\]

\begin{rem}
\label{rem:Starting-with-a}Starting from a $\text{BC}(\mathbf{S})$,
it is possible to normalize this configuration so that the new configuration
is a $\text{BC}(\mathbf{S})$ in $\mathcal{N}(\mathbf{S})$. Indeed,
assume that $\mathbf{q}$ is a $\text{BC}(\mathbf{S})$ with center
of mass $\mathbf{c}$, i.e. the equation 
\[
\underbrace{\sum_{\substack{j=1\\
j\not=i
}
}^{n}\frac{m_{i}m_{j}}{||\mathbf{q}_{j}-\mathbf{q}_{i}||^{3}}(\mathbf{q}_{j}-\mathbf{q}_{i})}_{=\nabla_{i}U_{n}(\mathbf{q})}+m_{i}\lambda\mathbf{S}(\mathbf{q}_{i}-\mathbf{c})=\boldsymbol{0},
\]
is satisfied for all $i=1,\ldots,n$. The configuration $\mathbf{q}$
can be normalized to $\widetilde{\mathbf{q}}\in\mathcal{N}(\mathbf{S})$
by means of the following procedure. Setting $\widetilde{\mathbf{q}}_{i}=\sqrt{1/I_{\mathbf{S}}(\mathbf{q})}(\mathbf{q}_{i}-\mathbf{c})$,
we find 
\[
\nabla_{i}U_{n}(\widetilde{\mathbf{q}})=I_{\mathbf{S}}(\mathbf{q})\nabla_{i}U_{n}(\mathbf{q}).
\]
Consequently, we obtain 
\begin{align*}
\boldsymbol{0}=I_{\mathbf{S}}(\mathbf{q})[\nabla_{i}U_{n}(\mathbf{q})+m_{i}\lambda\mathbf{S}(\mathbf{q}_{i}-\mathbf{c})] & =\nabla_{i}U_{n}(\widetilde{\mathbf{q}})+m_{i}\underbrace{I_{\mathbf{S}}(\mathbf{q})^{\frac{3}{2}}\lambda}_{=\widetilde{\lambda}(\widetilde{\mathbf{q}})=\widetilde{\lambda}}\mathbf{S}\widetilde{\mathbf{q}}_{i},
\end{align*}
and further, 
\[
\frac{U_{n}(\widetilde{\mathbf{q}})}{I_{\mathbf{S}}(\widetilde{\mathbf{q}})}=\widetilde{\lambda}=I_{\mathbf{S}}(\mathbf{q})^{\frac{3}{2}}\lambda=I_{\mathbf{S}}(\mathbf{q})^{\frac{1}{2}}U_{n}(\mathbf{q}).
\]
On the other hand, we have 
\begin{align*}
\widetilde{\mathbf{c}} & =\Bigl(\sum_{j=1}^{n}m_{j}\Bigr)^{-1}\Bigl(\sum_{j=1}^{n}m_{j}\widetilde{\mathbf{q}}_{j}\Bigr)\\
 & =\Bigl(\sum_{j=1}^{n}m_{j}\Bigr)^{-1}\Bigl(\sum_{j=1}^{n}m_{j}\sqrt{\frac{1}{I_{\mathbf{S}}(\mathbf{q})}}(\mathbf{q}_{j}-\mathbf{c})\Bigr)\\
 & =\sqrt{\frac{1}{I_{\mathbf{S}}(\mathbf{q})}}\Bigl(\sum_{j=1}^{n}m_{j}\Bigr)^{-1}\underbrace{\Bigl(\sum_{j=1}^{n}m_{j}(\mathbf{q}_{j}-\mathbf{c})\Bigr)}_{=\boldsymbol{0}}\\
 & =\boldsymbol{0},
\end{align*}
and 
\[
I_{\mathbf{S}}(\widetilde{\mathbf{q}})=\sum_{j=1}^{n}m_{j}\widetilde{\mathbf{q}}_{j}^{T}\mathbf{S}\widetilde{\mathbf{q}}_{j}=\frac{1}{I_{\mathbf{S}}(\mathbf{q})}\sum_{j=1}^{n}m_{j}(\mathbf{q}_{j}-\mathbf{c})^{T}\mathbf{S}(\mathbf{q}_{j}-\mathbf{c})=1.
\]
Thus, $\widetilde{\mathbf{q}}$ is a $\text{BC}(\mathbf{S})$ with
the parameter $\widetilde{\lambda}=U_{n}(\widetilde{\mathbf{q}})>0$,
and has the center of mass $\widetilde{\mathbf{c}}=0$ and the $\mathbf{S}$-weighted
moment of inertia $I_{\mathbf{S}}(\widetilde{\mathbf{q}})=1$. Therefore,
$\widetilde{\mathbf{q}}\in\text{\ensuremath{\mathcal{N}(\mathbf{S})}}$.
Explicitly this means, 
\begin{equation}
\sum_{\substack{j=1\\
j\not=i
}
}^{n}\frac{m_{j}}{||\widetilde{\mathbf{q}}_{j}-\widetilde{\mathbf{q}}_{i}||^{3}}(\widetilde{\mathbf{q}}_{j}-\widetilde{\mathbf{q}}_{i})+U_{n}(\widetilde{\mathbf{q}})\mathbf{S}\widetilde{\mathbf{q}}_{i}=\boldsymbol{0},\label{eq:A1}
\end{equation}
for $i=1,\ldots,n$, and 
\begin{equation}
\sum_{j=1}^{n}m_{j}\widetilde{\mathbf{q}}_{j}^{T}\mathbf{S}\widetilde{\mathbf{q}}_{j}=1.\label{eq:A2}
\end{equation}
\end{rem}

Let $V:\mathcal{N}(\mathbf{S})\rightarrow\mathbb{R}$ be the restriction
of $U_{n}$ to the manifold $\mathcal{N}(\mathbf{S})$. From Moeckel
(2014a) we have the following result. 
\begin{prop}
\label{prop:Let--be}Assume that $\mathbf{S}$ is a positive definite
symmetric $2\times2$ matrix. Then a configuration $\mathbf{q}$ is
a \emph{$\text{BC}(\mathbf{S})$} if and only if its corresponding
normalized configuration $\widetilde{\mathbf{q}}\in\mathcal{N}(\mathbf{S})$
(as in Remark \ref{rem:Starting-with-a}) is a critical point of $\widetilde{U}_{n}=U_{n}|_{\mathcal{N}(\mathbf{S})}:\mathcal{N}(\mathbf{S})\rightarrow\mathbb{R}$. 
\end{prop}

Proposition \ref{prop:Let--be} enlightens the fact that the dilation
freedom of finding balanced configurations, discussed in Remark \ref{rem:Lemma--treats},
can be suppressed by looking for balanced configurations as critical
points of $U_{n}$ on the manifold $\mathcal{N}(\mathbf{S})$. However,
the solutions of balanced configurations (or central configurations)
on $\mathcal{N}(\mathbf{S})$ may not be isolated. In this case we
introduce a notion of non-degenerate critical point.

For $\widetilde{\mathbf{q}}\in\text{Crit}(\widetilde{U}_{n})$ there
exists a Hessian quadratic form on $T_{\widetilde{\mathbf{q}}}\mathcal{N}(\mathbf{S})$
given locally by a symmetric matrix $\mathbf{H}(\widetilde{\mathbf{q}})=\mathbf{v}^{T}D^{2}\widetilde{U}_{n}(\widetilde{\mathbf{q}})\mathbf{v}$
for $\mathbf{v}\in T_{\widetilde{\mathbf{q}}}\mathcal{N}(S)$. The
nullity at a critical point is defined as $\text{null}(\widetilde{\mathbf{q}}):=\dim(\ker(H(\widetilde{\mathbf{q}})))$.
Instead of working in local coordinates on the manifold $\mathcal{N}(\mathbf{S})$,
we represent the Hessian by a $N\times N$ matrix (also called $\mathbf{H}(\widetilde{\mathbf{q}})$),
whose restriction to $T_{\widetilde{\mathbf{q}}}\mathcal{N}(\mathbf{S})$
gives the correct values. From Moeckel (2014a), the Hessian of $\widetilde{U}_{n}:\mathcal{N}(\mathbf{S})\rightarrow\mathbb{R}$
at a critical point $\widetilde{\mathbf{q}}\in\text{Crit}(\widetilde{U}_{n})$
is given by $\mathbf{H}(\widetilde{\mathbf{q}})\mathbf{v}=\mathbf{v}^{T}\mathbf{H}(\widetilde{\mathbf{q}})\mathbf{v}$,
where 
\begin{equation}
\mathbf{H}(\widetilde{\mathbf{q}})=D^{2}U_{n}(\widetilde{\mathbf{q}})+U_{n}(\widetilde{\mathbf{q}})\widehat{\mathbf{S}}\mathbf{M}\label{eq:A3}
\end{equation}
and 
\[
\widehat{\mathbf{S}}=\underbrace{\left(\begin{array}{ccc}
\mathbf{S} & \cdots & 0\\
\vdots & \ddots & \vdots\\
0 & \cdots & \mathbf{S}
\end{array}\right)}_{n\text{ blocks}},\quad\mathbf{M}=\left(\begin{array}{cccccc}
m_{1}\\
 & m_{1} &  &  & \boldsymbol{0}\\
 &  &  & \ddots\\
 & \boldsymbol{0} &  &  & m_{n}\\
 &  &  &  &  & m_{n}
\end{array}\right).
\]

If $\sigma_{x}=\sigma_{y}$, the normalized configuration space $\mathcal{N}(\mathbf{S})$
carries an $\text{O}(2)$-action. Hence, since $\widetilde{U}_{n}$
is $\text{O}(2)$-invariant it descends to a function $\widehat{U}_{n}:\mathcal{N}(\mathbf{S})/\text{O}(2)\rightarrow\mathbb{R}$.
For $\widetilde{\mathbf{q}}\in\mathcal{N}(\mathbf{S})$, we have 
\[
T_{[\widetilde{\mathbf{q}}]}\left(\mathcal{N}(\mathbf{S})/\text{O}(2)\right)\cong T_{\widetilde{\mathbf{q}}}\mathcal{N}(\mathbf{S})/T_{\widetilde{\mathbf{q}}}(\text{O}(2)\widetilde{\mathbf{q}}),
\]
where $[\widetilde{\mathbf{q}}]\in\mathcal{N}(\mathbf{S})/\text{O}(2)$
represents the equivalence class of $\widetilde{\mathbf{q}}$. If
$\widetilde{\mathbf{q}}\in\mathcal{N}(\mathbf{S})$ is a critical
point of $\widetilde{U}_{n}$, then $\text{null}(\widetilde{\mathbf{q}})\geq1$
and the equivalence class $[\widetilde{\mathbf{q}}]\in\mathcal{N}(\mathbf{S})/\text{O}(2)$
is a critical point of $\widehat{U}_{n}$. Then the Hessian $\widehat{\mathbf{H}}([\widehat{\mathbf{q}}])$
of $\widehat{U}_{n}$ at $[\widetilde{\mathbf{q}}]$ is obtained by
descending $\mathbf{H}(\widetilde{\mathbf{q}})$ to the space $T_{[\widetilde{\mathbf{q}}]}\left(\mathcal{N}(\mathbf{S})/\text{O}(2)\right)$.
More precisely, if $[\mathbf{v}],[\mathbf{w}]\in T_{[\widetilde{\mathbf{q}}]}\left(\mathcal{N}(\mathbf{S})/\text{O}(2)\right)$,
where $\mathbf{v},\mathbf{w}\in T_{\widetilde{\mathbf{q}}}\mathcal{N}(\mathbf{S})$,
then $\widehat{\mathbf{H}}([\widetilde{\mathbf{q}}])([\mathbf{v}],[\mathbf{w}])=\mathbf{H}(\widetilde{\mathbf{q}})(\mathbf{v},\mathbf{w})$. 

The non-degeneracy of a critical point is defined as follows. 
\begin{defn}
\label{def:Let-.-A}Let $\mathbf{q}$ be a $\text{BC}(\mathbf{S})$
with $\mathbf{S}=\text{diag}(\sigma_{x},\sigma_{y})$. The configuration
$\mathbf{q}$ is called non-degenerate if one of the following cases
hold. 
\begin{lyxlist}{00.00.0000}
\item [{Case$\ $1:}] If $\sigma_{x}=\sigma_{y}$, then the Hessian $\widehat{\mathbf{H}}([\widetilde{\mathbf{q}}])$
is non-degenerate, where $\widetilde{\mathbf{q}}$ represents the
corresponding normalized configuration of $\mathbf{q}$. 
\item [{Case$\ $2:}] If $\sigma_{x}\not=\sigma_{y}$, then the Hessian
$\mathbf{H}(\widetilde{\mathbf{q}})$ is non-degenerate, where $\widetilde{\mathbf{q}}$
represents the corresponding normalized configuration of $\mathbf{q}$. 
\end{lyxlist}
\end{defn}

It should be pointed out that if $\mathbf{\widetilde{q}}=(\mathbf{\widetilde{q}}_{1}^{T},\ldots,\mathbf{\widetilde{q}}_{n}^{T})^{T}$
is a (normalized) central configuration, then the mutual distances
$R_{ij}=||\widetilde{\mathbf{q}}_{j}-\widetilde{\mathbf{q}}_{i}||$
satisfy the Albouy-Chenciner equations (Albouy and Chenciner 1998),
\begin{equation}
f_{ij}(\mathbf{R}):=\sum_{k=1}^{n}m_{k}[S_{ik}(R_{jk}^{2}-R_{ik}^{2}-R_{ij}^{2})+S_{jk}(R_{ik}^{2}-R_{jk}^{2}-R_{ij}^{2})]=0,\label{eq:ALBOUYEq}
\end{equation}
for all $1\leq i<j\leq n$, where $\mathbf{R}=(R_{12},R_{13},\ldots,R_{n-1.n})^{T}$
and
\begin{align}
S_{ij} & =\left\{ \begin{array}{c}
1/R_{ij}^{3}+\lambda^{\prime},\\
0,
\end{array}\begin{array}{c}
i\neq j\\
i=j
\end{array}\right.,\label{eq:ALBOUY1}\\
\lambda^{\prime} & =-U_{n}(\widetilde{\mathbf{q}})/\sum_{k=1}^{n}m_{k}.\label{eq:ALBOUY2}
\end{align}
Conversely, if the quantities $R_{ij}$ are the mutual distances of
some configuration $\mathbf{\widetilde{q}}$ and they satisfy Eq.
(\ref{eq:ALBOUYEq}), then the configuration is central (Albouy and
Chenciner 1998). For a detailed derivation of the Albouy-Chenciner
equations we refer to Albouy and Chenciner (1998), and Hampton and
Moeckel (2006). 

In the forthcoming analysis, we will consider only normalized configurations
and renounce on the tilde character (\textasciitilde ). 

\section{Stochastic optimization algorithm}

Consider the system of nonlinear equations 
\begin{equation}
\mathbf{f}(\mathbf{q})=\boldsymbol{0}\label{eq:MF1}
\end{equation}
with $\mathbf{q}\in B=[a_{1},b_{1}]\times[a_{2},b_{2}]\ldots\times[a_{N},b_{N}]\subset\mathbb{R}^{N}$,
$\mathbf{f}(\mathbf{q})=(f_{1}(\mathbf{q}),f_{2}(\mathbf{q}),\ldots,f_{M}(\mathbf{q}))^{T}\in\mathbb{R}^{M}$,
$M\geq N$, and $f_{i}:B\rightarrow\mathbb{R}$ being continuous functions.
The system of equations (\ref{eq:MF1}) can be transformed into a
nonlinear least-squares problem by defining the objective function
\[
F(\mathbf{q})=\frac{1}{2}||\mathbf{f}(\mathbf{q})||^{2}.
\]
Since by construction, $F(\mathbf{q})\geq0$, we infer that a global
minimum $\mathbf{q}^{\star}$ of $F(\mathbf{q})$ satisfies $F(\mathbf{q}^{\star})=0$
and consequently, $\mathbf{f}(\mathbf{q}^{\star})=0$; thus, $\mathbf{q}^{\star}$
is a root of the corresponding system of equations. Finding all (global)
minima $\mathbf{q}^{\star}$ with $F(\mathbf{q}^{\star})=0$ corresponds
to locating all the roots of the system. If some local minima of $F(\mathbf{q})$
will have function value greater than zero, they will be discarded
since they do not correspond to the roots of the system.

The task of locating all local minima of a multidimensional continuous
differentiable function $F:B\subset\mathbb{R}^{N}\rightarrow\mathbb{R}$,
may be defined as follows: Find all $\mathbf{q}_{i}^{\star}\in B\subset\mathbb{R}^{N}$
that satisfy 
\[
\mathbf{q}_{i}^{\star}=\textrm{arg}\min_{\mathbf{q}\in B_{i}}F(\mathbf{q}),\,\,\,B_{i}=B\cap\{\mathbf{q}\mid|\mathbf{q}-\mathbf{q}_{i}^{\star}|<\varepsilon\}.
\]
Among several methods dealing with this problem, stochastic methods
are the most popular due to their effectiveness and implementation
simplicity. The widely used stochastic method is Multistart. In Multistart
a point is sampled uniformly from the feasible region, and subsequently
a local search is started from it. The weakness of this algorithm
is that the same local minima can be repeatedly found, wasting so
computational resources. For this reason, clustering methods that
attempt to avoid this drawback have been developed (Becker and Lago
1970; Törn 1978; Boender et al. 1982; Kan and Timmer 1987a; Kan and
Timmer 1987b). A cluster is defined as a set of points that are assumed
to belong to the region of attraction of the same minimum, and so,
only one local search is required to locate it. The region of attraction
of a local minimum $\mathbf{q}^{\star}$ is defined as 
\[
A(\mathbf{q}^{\star})=\{\mathbf{q}\mid\mathbf{q}\in B\subset\mathbb{R}^{N},\,\,\,\mathcal{L}(\mathbf{q})=\mathbf{q}^{\star}\},
\]
where $\mathcal{L}(\mathbf{q})$ is the point where the local search
procedure $\mathcal{L}$ terminates when started at point $\mathbf{q}$.
Here, $\mathcal{L}$ is a deterministic local optimization method
such as BFGS (Fletcher 1970), Steepest Descent, Modified Newton, etc.
Representative for clustering techniques is the Minfinder method of
Tsoulos and Lagaris (2006). This method, which is illustrated in Algorithm
\ref{alg:AlgMinFinder}, relies on the Topographical Multilevel Single
Linkage of Ali and Storey (1994), and is the heart of our stochastic
optimization method. However, the method has some additional features
that serve our purpose. The following key elements are apparent in
Algorithm \ref{alg:AlgMinFinder}: 
\begin{enumerate}
\item the selection of a starting point for the local search, 
\item the generation of sampling points, 
\item the local optimization method,
\item the specification of the set of distinct solutions, and
\item the stopping rule. 
\end{enumerate}
Their description is laid out in the following subsections.

\begin{algorithm}
\caption{Minfinder method. \label{alg:AlgMinFinder}}

$\bullet$ Initialize the set of distinct solutions $Q=\textrm{Ø}$.

$\bullet$ Generate a set $S=\{\mathbf{s}_{i}\}_{i=1}^{N_{\textrm{s}}}$
of $N_{\textrm{s}}$ sample points in the box $B$.

\textbf{For }all $\mathbf{s}\in S$ \textbf{do}

\hspace{0.5cm}\textbf{If} $\mathbf{s}$ is a start point \textbf{then}

\hspace{1cm}$\bullet$ Start a local search $\mathbf{q}=\mathcal{L}(\mathbf{s})$.

\hspace{1cm}$\bullet$ If $\mathbf{q}\notin Q$, insert $\mathbf{q}$
in the set of distinct solutions $Q$.

\hspace{1cm}\textbf{End if}

\hspace{0.5cm}\textbf{If }a stopping rule applies, \textbf{exit}.

\textbf{End for}
\end{algorithm}

Before proceeding we note that in our case, the dimensions of the
optimization problem are $M=N=2n$, where $n$ is the number of masses.
The vector $\mathbf{q}$ stands for a configuration, that is, $\mathbf{q}$
includes the Cartesian coordinates of the point masses, 
\[
\mathbf{q}=((x_{1},y_{1}),\ldots,(x_{n},y_{n}))^{T}\in\mathbb{R}^{N}\backslash\Delta.
\]
For simplicity we restrict ourselves to the case of equal masses,
i.e., $m_{i}=m$ for all $i=1,\ldots,n$. According to Eq. (\ref{eq:A1})
and for $\mathbf{S}=\text{diag}(\sigma_{x},\sigma_{y})$, the functions
$f_{i}(\mathbf{q})$ that determine the objective function $F(\mathbf{q})$
are 
\begin{align}
f_{2i-1}(\mathbf{q}) & =\sum_{j=1,j\neq i}^{n}m\frac{x_{j}-x_{i}}{R_{ij}^{3}}+U_{n}(\mathbf{q})\sigma_{x}x_{i},\label{eq:Fx}\\
f_{2i}(\mathbf{q}) & =\sum_{j=1,j\neq i}^{n}m\frac{y_{j}-y_{i}}{R_{ij}^{3}}+U_{n}(\mathbf{q})\sigma_{y}y_{i},\label{eq:Fy}
\end{align}
for $i=1,\ldots,n$, where $U_{n}\bigl(\mathbf{q}\bigr)=\sum_{1\leq i<j\leq n}m^{2}/R_{ij}$
with $R_{ij}=||\mathbf{q}_{i}-\mathbf{q}_{j}||$. In view of Eq. (\ref{eq:A2}),
that is,
\begin{equation}
\sum_{i=1}^{n}m\bigl(\sigma_{x}x_{i}^{2}+\sigma_{y}y_{i}^{2}\bigr)=1,\label{eq:Fc}
\end{equation}
 the following simple bounds on the variables 
\begin{align}
 & -\frac{1}{\sqrt{m\sigma_{x}}}\leq x_{i}\leq\frac{1}{\sqrt{m\sigma_{x}}},\label{eq:Bx}\\
 & -\frac{1}{\sqrt{m\sigma_{y}}}\leq y_{i}\leq\frac{1}{\sqrt{m\sigma_{y}}}\label{eq:By}
\end{align}
are imposed. Thus, we are faced with the solution of a nonlinear least-squares
problem with simple bounds on the variables.

\subsection{Selection of a starting point}

In the Minfinder algorithm, a point is considered to be a start point
if it is not too close to some already located minimum or another
sample, whereby the closeness with a local minimum or some other sample
is guided through the so-called typical distance and the gradient
criterion. In our implementation, we disregard the gradient criterion
because we intend to capture as many solutions as possible. In this
regard, according to Tsoulos and Lagaris (2006), a point $\mathbf{s}$
is considered as start point if none of the following conditions holds.
\begin{enumerate}
\item There is an already located minimum $\mathbf{q}\in Q$ that satisfies
the condition 
\begin{align}
\textrm{C}1: & \,\,\,||\mathbf{s}-\mathbf{q}||<d_{\min}(Q),\,\,\,d_{\min}(Q)=\min_{i,j;i\neq j}||\mathbf{q}_{i}-\mathbf{q}_{j}||,\,\,\,\mathbf{q}_{i},\mathbf{q}_{j}\in Q.\label{eq:C1}
\end{align}
\item $\mathbf{s}$ is near to another sample point $\mathbf{s}'\in S$
that satisfies the condition 
\begin{align}
\textrm{C}2: & \,\,\,||\mathbf{s}-\mathbf{s}'||<r_{\textrm{t}}.\label{eq:C3}
\end{align}
\end{enumerate}
In Eq. (\ref{eq:C3}), $r_{\textrm{t}}$ is a typical distance defined
by 
\begin{align}
r_{\textrm{t}} & =\frac{1}{L}R_{\mathrm{t}},\,\,\,R_{\mathrm{t}}=\sum_{i=1}^{L}||\mathbf{s}_{i}-\mathcal{L}(\mathbf{s}_{i})||,\label{eq:Typicaldistance}
\end{align}
where $\mathbf{s}_{i}$ are starting points for the local search procedure
$\mathcal{L}$, and $L$ is the number of performed local searches.
The main idea behind Eq. (\ref{eq:Typicaldistance}) is that after
a number of iterations and a number of local searches, the quantity
$r_{\textrm{t}}$ will be a reasonable approximation for the mean
radius of the regions of attraction.

\subsection{Generation of sampling points}

A powerful sampling method should create data that accurately represents
the underlying function preserving the statistical characteristics
of the complete dataset. The following sampling methods are implemented
in our algorithm. 
\begin{enumerate}
\item Pseudo-Random Number Generators\emph{ }(Marsaglia and Tsang 2000;
Matsumoto and Nishimura 1998). These methods attempt to generate statistically
uniform random numbers within the given range. They are fast and simple
to use, but the samples are not distributed uniformly enough especially
for the cases of a low number of sampling points and/or large number
of dimensions. 
\item Chaotic Methods (Dong et al. 2012; Gao and Wang 2007; Gao and Liu
2012). Theoretically, chaotic motion can traverse every state in a
certain region. As compared to Pseudo-Random Number Generators, chaotic
initialization methods can form better distributions in the search
space due to the randomness and non-repetitivity of chaos. 
\item Low Discrepancy Methods. These methods have the support of theoretical
upper bounds on discrepancy (i.e., non-uniformity) and belong to the
category of deterministic methods (no randomness is involved in their
algorithms). Uniform populations are created using quasi-random or
sub-random sequences that cover the input space quickly and evenly,
while the uniformity and coverage improve continually as more data
points are added to the sequence. Halton, Sobol, Niederreiter, Hammersley,
and Faure are well known sequences from this category. These work
well in low dimensions, but they lose uniformity in high dimensions. 
\item Latin Hypercube (McKay et al. 1979). Latin hypercube sampling partitions
the input space into bins of equal probability and distributes the
samples in such a way that only one sample is located in each axis-aligned
hyperplane. The method is useful when the underlying function has
a low order distribution but produces clustering of sampling points
at high dimensions. 
\item Quasi-Oppositional Differential Evolution (Rahnamayan et al. 2006;
Rahnamayan et al. 2008). The method can be regarded as a two-step
method. The algorithm (i) generates a random original population,
(ii) determines the opposition points of the original population,
(iii) merges both populations into one big population, and finally,
(iv) selects the best individuals according to the lowest values of
the objective function. 
\item Centroidal Voronoi Tessellation (Du et al. 2010). Centroidal Voronoi
Tessellation produces sample points located at the mass center of
each Voronoi cell covering the input space. Actually, the algorithm
starts with an initial partition and then iteratively updates the
estimate of the centroids of the corresponding Voronoi subregions.
The initial population can be generated, for example, by Pseudo-Random
Number Generators or Low Discrepancy methods. One drawback of the
method lies in being computationally demanding for high dimensional
spaces. 
\end{enumerate}

\subsection{Local optimization methods}

There is an impressive number of optimization software packages for
nonlinear least squares and general function minimization (for a survey,
we refer to the monograph by More and Wright (1993)). The following
optimization algorithms are implemented in our computer code. 
\begin{enumerate}
\item The BFGS algorithm of Byrd, Nocedal and Zhu (1995). The algorithm
relies on the gradient projection method to determine a set of active
constraints at each iteration, and uses a line search procedure to
compute the step length, as well as, a limited memory BFGS matrix
to approximate the Hessian of the objective function. 
\item The TOLMIN algorithm of Powell (1989). The algorithm includes (i)
quadratic approximations of the objective function whose second derivative
matrices are updated by means of the BFGS formula, (ii) active sets
technique, and (iii) a primal-dual quadratic programming procedure
for calculation of the search direction. Each search direction is
calculated so that it does not intersect the boundary of any inequality
constraint that is satisfied and that has a ``small'' residual at
the beginning of the line search.
\item The DQED algorithm due to Hanson and Krogh (1992). The algorithm is
based on approximating the nonlinear functions using the quadratic-tensor
model proposed by Schnabel and Frank (1984). The objective function
is allowed to increase at intermediate steps, as long as a predictor
indicates that a new set of best values exists in a trust-region (defined
by a box containing the current values of the unknowns). 
\item The optimization algorithms implemented in the Portable, Outstanding,
Reliable and Tested (PORT) library. The algorithms use a trust-region
method in conjunction with a Gauss-Newton and a Quasi-Newton model
to compute the trial step (Dennis Jr. et al. 1981a; Dennis Jr. et
al. 1981b). When the first trial steps fails, the alternate model
gets a chance to make a trial step with the same trust-region radius.
If the alternate model fails to suggest a more successful step, then
the current model is maintained for the duration of the present iteration
step. The trust-region radius is then decreased until the new iterate
is determined or the algorithm fails. In particular, the routine DRN2GB
for nonlinear least squares and working in reverse communication is
used in our applications. Note that reverse-communication drivers,
return to their caller (e.g., the main program) whenever they need
to know $\mathbf{f}(\mathbf{q})$ and/or $\partial\mathbf{f}(\mathbf{q})/\partial\mathbf{q}$
at a new $\mathbf{q}$. The calling routine must then compute the
necessary information and call the reverse-communication driver again,
passing it the information it wants. 
\end{enumerate}
In some applications, the computation of the Jacobian matrix $\partial\mathbf{f}(\mathbf{q})/\partial\mathbf{q}$
is more time consuming than the computation of $\mathbf{f}(\mathbf{q})$.
To reduce the number of calls to the derivative routine, the above
deterministic optimization algorithms are used in conjunction with
some stochastic solvers, as for example, (i) evolutionary strategy,
(ii) genetic algorithms, and (iii) simulated annealing. Essentially,
the algorithm is organized so that the output delivered by a stochastic
algorithm is the input (initial guess) of a deterministic algorithm.
It should be point out that the user has the option to use the stochastic
algorithms in a stand-alone mode.

\subsection{Set of distinct solutions}

For a central configuration ($\sigma_{x}=\sigma_{y}$), if 
\[
\mathbf{q}=((x_{1},y_{1}),\ldots,(x_{n},y_{n}))^{T}
\]
is a solution, then any (i) permuted solution 
\[
\mathcal{P}\mathbf{q}=((x_{\sigma(1)},y_{\sigma(1)}),\ldots,(x_{\sigma(n)},y_{\sigma(n)}))^{T},
\]
(ii) rotated solution 
\[
\mathcal{R}_{\alpha}\mathbf{q}=((x_{1}^{\prime},y_{1}^{\prime}),\ldots,(x_{n}^{\prime},y_{n}^{\prime}))^{T}\textrm{ with }\left(\begin{array}{c}
x_{i}^{\prime}\\
y_{i}^{\prime}
\end{array}\right)=\mathbf{R}(\alpha)\left(\begin{array}{c}
x_{i}\\
y_{i}
\end{array}\right),
\]
where $\mathbf{R}(\alpha)$ is a rotation matrix of angle $\alpha$,
(iii) conjugated solution $\mathcal{C}\mathbf{q}$, where $\mathcal{C}\mathbf{q}$
stands for the reflected solutions with respect to the $x$- and $y$-axis,
i.e., 
\[
\mathcal{C}_{x}\mathbf{q}=((x_{1},-y_{1}),\ldots,(x_{n},-y_{n}))^{T},
\]
and 
\[
\mathcal{C}_{y}\mathbf{q}=((-x_{1},y_{1}),\ldots,(-x_{n},y_{n}))^{T}
\]
respectively, are also solutions. For a balanced configuration, if
$\mathbf{q}$ is a solution, then (i) any permuted solution, (ii)
a solution rotated by $\alpha=\pi$, and (iii) any conjugated solution
are also solutions. Taking these results into account, we define the
set of distinct solutions $Q$ as follows. First, we consider the
set of all solutions 
\[
Q_{0}=\{\mathbf{q}\in\mathbb{R}^{N}\backslash\Delta\mid\mathbf{q}\text{ solves the optimization problem}\}.
\]
Then, for central configurations, we introduce an equivalence relation
according to which two elements $\mathbf{q}_{1},\mathbf{q}_{2}\in Q_{0}$
are called equivalent (in notation $\mathbf{q}_{1}\sim\mathbf{q}_{2}$)
if and only if one of the following conditions are satisfied: (i)
there exists a permutation $\mathcal{P}$ such that $\mathbf{q}_{1}=\mathcal{P}\mathbf{q}_{2}$,
(ii) there exists an angle $\alpha$ such that $\mathbf{q}_{1}=\mathcal{R}_{\alpha}\mathbf{q}_{2}$,
(iii) $\mathbf{q}_{1}=\mathcal{C}_{x}\mathbf{q}_{2}$, and (iv) $\mathbf{q}_{1}=\mathcal{C}_{y}\mathbf{q}_{2}$.
For balanced configurations, the equivalence relation is: $\mathbf{q}_{1}\sim\mathbf{q}_{2}$,
where $\mathbf{q}_{1},\mathbf{q}_{2}\in Q_{0}$, if and only if one
of the following conditions are satisfied: (i) there exists a permutation
$\mathcal{P}$ such that $\mathbf{q}_{1}=\mathcal{P}\mathbf{q}_{2}$,
(ii) $\mathbf{q}_{1}=\mathcal{R}_{\pi}\mathbf{q}_{2}$, (iii) $\mathbf{q}_{1}=\mathcal{C}_{x}\mathbf{q}_{2}$,
and (iv) $\mathbf{q}_{1}=\mathcal{C}_{y}\mathbf{q}_{2}$. In both
cases the set of distinct solutions $Q$ is $Q=Q_{0}/\sim$. 

An important step of the algorithm is to test if a solution $\mathbf{q}$,
computed by means of a local optimization method, i.e., $\mathbf{q}\in Q_{0}$,
belongs to the set of (distinct) solutions $Q=\{\mathbf{q}_{i}\}_{i=1}^{N_{\textrm{sol}}}$;
otherwise, $\mathbf{q}$ will be included in $Q$. For doing this,
we use the following procedure: if (i) the objective function at $\mathbf{q}$
is smaller than a prescribed tolerance and (ii) the ordered set of
mutual distances $\{R_{ij}\}$ corresponding to $\mathbf{q}$ does
not coincides with the ordered set of mutual distances $\{R_{ij}^{\prime}\}$
corresponding to any $\mathbf{q}'\in Q$, then $\mathbf{q}$ is inserted
in the set of solutions $Q$. 

\subsection{Stopping rules\label{subsec:Stopping-rules}}

A reliable stopping rule should terminate the iterative process when
all minima have been collected with certainty. Several Bayesian stopping
rules make use (i) of estimates of the fraction of the uncovered space
(Zieli\'{n}ski 1981) and the number of local minima (Boender and Kan
1987), or (ii) on the probability that all local minima have been
observed (Boender and Romeijn 1995). For functions with many local
minima, these stopping rules are not very efficient, because for example,
in some cases, the number of local searches must be greater than the
square of the located minima. More effective termination criteria
based on asymptotic considerations have been designed by Lagaris and
Tsoulos (2008). These include (i) the Double-Box stopping rule, which
uses a Monte Carlo based model that enables the determination of the
coverage of the bounded search domain, (ii) the Observables stopping
rule, which relies on a comparison between the expectation values
of observable quantities to the actually measured ones, and (iii)
the Expected Minimizers stopping rule, which is based on estimating
the expected number of local minima in the specified domain.

The Double-Box stopping rule can be summarized as follows. Choose
the integers $K$ and $N_{\textrm{s}}^{0}$. Let $B_{2}$ be a larger
box that contains $B$ such that $\mu(B_{2})=2\mu(B)$, where $\mu(B)$
is the measure of $B$. At every iteration $k$, where $1\leq k\leq K$,
sample $B_{2}$ uniformly until $N_{\textrm{s}}^{0}$ points fall
in $B$. After $k$ iterations, let $M_{k}$ be the accumulated (total)
number of points from $B_{2}$. Then, the quantity $\delta_{k}=kN_{\textrm{s}}^{0}/M_{k}$
has an expectation value $\bigl\langle\delta\bigr\rangle_{(k)}=(1/k)\sum_{i=1}^{k}\delta_{i}$
that tends to $\mu(B)/\mu(B_{2})=1/2$ as $k\rightarrow\infty$, while
the variance $\sigma_{(k)}^{2}(\delta)=(1/k)\sum_{i=1}^{k}(\delta_{i}-\bigl\langle\delta\bigr\rangle_{(k)})^{2}$
tends to zero as $k\rightarrow\infty$. The variance is a smoother
quantity than the expectation, and is better suited for a termination
criterion. Actually, the iterative process is stopped when the variance
$\sigma_{(k)}^{2}(\delta)$ is below a prescribed tolerance. 

From the above discussion it is apparent that the Double-Box stopping
rule requires a specific sampling. In order to implement all sampling
methods into a common framework, we generate a set $S=\{\mathbf{s}_{i}\}_{i=1}^{N_{\textrm{s}}}$
of $N_{\textrm{s}}=KN_{\textrm{s}}^{0}$ sample points in the box
$B$, where $K$ is the number of disjoint subsets in which the set
$S$ is split, i.e., $S=\cup_{k=1}^{K}S_{k}$, and $N_{\textrm{s}}^{0}$
the number of sample points in each subset $S_{k}$. Because in some
applications, the requirement of finding as many solutions as possible
asks for a very small tolerance of the Double-Box stopping rule (and
so, for a large number of iterations) we use an alternative termination
criterion: if the number of solutions $N_{\textrm{sol}}(k)$ does
not change within $k^{\star}\leq K$ iteration steps, i.e., $N_{\textrm{sol}}(k)=N_{\textrm{sol}}(l)$
for all $l=k-k^{\star}+1,\ldots,k-1$, the algorithm stops. Algorithm
\ref{alg:AlgStochastic} illustrates the main steps of the stochastic
optimization method used in this work.

\begin{algorithm}
\caption{The main steps of the stochastic optimization method.\label{alg:AlgStochastic}}

$\bullet$ Choose $K$ and $N_{\textrm{s}}^{0}$. Generate a set $S=\{\mathbf{s}_{i}\}_{i=1}^{N_{\textrm{s}}}$
of $N_{\textrm{s}}=KN_{\textrm{s}}^{0}$ sample

points in the box $B$. For the Double-Box stooping rule, compute

$\delta_{k}$ for all $k=1,\ldots,K$.

$\bullet$ Initialize the set of distinct solutions $Q=\textrm{Ø}$,
the number of local 

searches $L=0$, and the sum determining the typical distance 

$R_{\textrm{t}}=0$.

\textbf{For} $k=1,\ldots,K$ \textbf{do}

\hspace{0.5cm}$\bullet$ Initialize the number of distinct solutions
at iteration step $k$,

\hspace{0.5cm}$N_{\mathrm{sol}}(k)=|Q|$.

\hspace{0.5cm}$\bullet$ Compute $\bigl\langle\delta\bigr\rangle_{(k)}$
and $\sigma_{(k)}^{2}(\delta)$.

\hspace{0.5cm}$\bullet$ Initialize the set of all solutions at iteration
step $k$, $Q_{0}=\textrm{Ø}$.

\hspace{0.5cm}\textbf{For} $i=1,\ldots,N_{\textrm{s}}^{0}$ \textbf{do}

\hspace{1cm}$\mathbf{s}=\mathbf{s}_{i+(k-1)N_{\textrm{s}}^{0}}$.

\hspace{1cm}$\bullet$ Test if $\mathbf{s}$ is a start point using
as control parameters

\hspace{1cm}$L$, $r_{\textrm{t}}=R_{\mathrm{t}}/L$, and $d_{\min}(Q_{0})$.

\hspace{1cm}\textbf{If} $\mathbf{s}$ is a start point \textbf{then}

\hspace{1.5cm}$\bullet$ Start a local search $\mathbf{q}=\mathcal{L}(\mathbf{s})$.

\hspace{1.5cm}$\bullet$ Insert $\mathbf{q}$ in the set $Q_{0}$. 

\hspace{1.5cm}$\bullet$ Update $L\leftarrow L+1$ and $R_{\textrm{t}}\leftarrow R_{\textrm{t}}+||\mathbf{s}-\mathbf{q}||$,
and

\hspace{1.5cm}compute $d_{\min}(Q_{0})=\min_{i\neq j}||\mathbf{q}_{i}-\mathbf{q}_{j}||$
for all $\mathbf{q}_{i},\mathbf{q}_{j}\in Q_{0}$. 

\hspace{1cm}\textbf{End if}

\hspace{0.5cm}\textbf{End for}

\hspace{0.5cm}\textbf{For} each $\mathbf{q}\in Q_{0}$ \textbf{do}

\hspace{1cm}\textbf{If} $\mathbf{q}\notin Q$, insert $\mathbf{q}$
in the set of distinct solutions $Q$ and 

\hspace{1cm}update $N_{\mathrm{sol}}(k)\leftarrow N_{\mathrm{sol}}(k)+1$.

\hspace{0.5cm}\textbf{End for}

\hspace{0.5cm}\textbf{If }(Double-Box stooping rule) \textbf{then}

\hspace{1cm}\textbf{If $\sigma_{(k)}^{2}(\delta)<\varepsilon_{\mathrm{DB}}$,
exit}

\hspace{0.5cm}\textbf{else}

\hspace{1cm}\textbf{If $N_{\mathrm{sol}}(k)=N_{\mathrm{sol}}(l)$}
for all $l=k-k^{\star}+1,\ldots,k-1$, \textbf{exit}

\hspace{0.5cm}\textbf{End if}

\textbf{End for}
\end{algorithm}

\section{Testing solutions}

In the post-processing step it is straightforward to check
\begin{enumerate}
\item if for any configuration $\mathbf{q}=(\mathbf{q}_{1}^{T},...,\mathbf{q}_{n}^{T})^{T}$,
the center of mass of the configuration is located at the vertex $(0,0)$,
that is, if the condition $\sum_{i=1}^{n}m_{i}\mathbf{q}_{i}=\boldsymbol{0}$
is satisfied,
\item if for any configuration $\mathbf{q}=(\mathbf{q}_{1}^{T},...,\mathbf{q}_{n}^{T})^{T}$,
the $\mathbf{S}$-weighted moment of inertia is normalized to one,
that is, if the condition $\sum_{j=1}^{n}m_{j}\mathbf{q}_{j}^{T}\mathbf{S}\mathbf{q}_{j}=1$
is satisfied, and
\item if for any central configuration $\mathbf{q}=(\mathbf{q}_{1}^{T},...,\mathbf{q}_{n}^{T})^{T}$,
the residual of the Albouy-Chenciner equations defined as 
\[
\Delta=\sqrt{\sum_{1\leq i<j\leq n}f_{ij}^{2}(\mathbf{R})},
\]
where (cf. Eq. (\ref{eq:ALBOUYEq})) 
\[
f_{ij}(\mathbf{R})=m\sum_{k=1}^{n}[S_{ik}(R_{jk}^{2}-R_{ik}^{2}-R_{ij}^{2})+S_{jk}(R_{ik}^{2}-R_{jk}^{2}-R_{ij}^{2})],\,\,\,1\leq i<j\leq n
\]
and $R_{ij}=||\mathbf{q}_{i}-\mathbf{q}_{j}||$, is sufficiently small. 
\end{enumerate}
Two additional tests related to the fulfillment of the Morse equality
and the uniqueness of the solutions are listed below. 

\subsection{Morse equality}

A drawback of the algorithm is that there is no guarantee that in
a finite number of steps, all solutions are found. Some hope that
the set $Q$, at least for small $n$, is complete, comes from the
Morse equality: if $\widehat{U}_{n}:\mathcal{N}(S)/\text{SO}(2)\rightarrow\mathbb{R}$
is a Morse function, we have 
\begin{equation}
\sum_{k}(-1)^{k}\nu_{k}=\chi(\mathcal{N}(S)/\text{SO}(2))=(-1)^{n}(n-2)!,\label{eq:Morse}
\end{equation}
where $\nu_{k}$ are the number of critical points of index $k$ and
$\chi(\mathcal{N}(S)/\text{SO}(2))$ is the Euler characteristic of
$\mathcal{N}(S)/\text{SO}(2)$ (Ferrario 2002). Each central configuration
$\mathbf{q}=(\mathbf{q}_{1}^{T},...,\mathbf{q}_{n}^{T})^{T}$ gives
rise to central configurations $(\mathbf{q}_{\sigma(1)}^{T},...,\mathbf{q}_{\sigma(n)}^{T})^{T}$
for all $\sigma\in S_{n}$, where $S_{n}$ is the group of permutations
of $n$ elements. When the configuration has an axis of symmetry,
there are $n!/j(\mathbf{q})$ many distinct critical points, where
$j(\mathbf{q})$ is the size of the isotropy group of the central
configuration $\mathbf{q}$. Otherwise, when the configuration has
no axis of symmetry, there are $n!/j(\mathbf{q})$ many distinct critical
points, as well as their reflections with respect to an axis in the
plane; hence, there are $2n!/j(\mathbf{q})$ distinct critical points.
Let $i(\mathbf{q})=j(\mathbf{q})$ when $\mathbf{q}$ has an axis
of symmetry and $i(\mathbf{q})=j(\mathbf{q})/2$ when $\mathbf{q}$
has no axis of symmetry. Furthermore, let $h(\mathbf{q})$ be the
Morse index of $\mathbf{q}$. Then, for central configurations, the
Morse equality (\ref{eq:Morse}) becomes
\begin{equation}
\sum_{i=1}^{N_{\textrm{sol}}}\frac{(-1)^{h(\mathbf{q}_{i})}}{i(\mathbf{q}_{i})}=\frac{(-1)^{n}}{n(n-1)}.\label{eq:Morse1}
\end{equation}
The quantities in Eq. (\ref{eq:Morse1}) are computed as follows. 
\begin{enumerate}
\item For a non-degenerate solution $\mathbf{q}$, the Morse index is given
by the number of negative eigenvalues of the Hessian matrix (cf. Eq.
(\ref{eq:A3})) 
\begin{equation}
\mathbf{H}(\mathbf{q})=D^{2}U_{n}(\mathbf{q})+U_{n}(\mathbf{q})\mathbf{S}\mathbf{M},\label{eq:Hessmatrix}
\end{equation}
where $D^{2}U_{n}(\mathbf{q})$ is computed as 
\begin{align*}
D^{2}U_{n}(\mathbf{q}) & =(\mathbf{D}_{ij})_{i,j=1}^{n}\in\mathbb{R}^{2n\times2n},\\
\mathbf{D}_{ij} & =\frac{m^{2}}{R_{ij}^{3}}(\mathbf{I}_{2\times2}-3\mathbf{u}_{ij}\mathbf{u}_{ij}^{T})\in\mathbb{R}^{2\times2},\,\,\,i\neq j,\\
\mathbf{D}_{ii} & =-\sum_{j=1}^{n}\mathbf{D}_{ij},\,\,\,i=j,\\
\mathbf{u}_{ij} & =(\mathbf{q}_{i}-\mathbf{q}_{j})/R_{ij}.
\end{align*}
\item The isotropy index is computed as in Ferrario (2002): we count all
configurations that are $\text{O}(2)$ invariant, that is, for a configuration
$\mathbf{q},$ we compute the isotropy index $i(\mathbf{q})$. Specifically,
for any collinear configuration, we take $i(\mathbf{q})=2$, while
for any non-collinear configuration, we take into account that this
type of configuration can have (i) as isotropy subgroup, a dihedral
group or the cyclic group of order 2, in which case, $i(\mathbf{q})$
is the number of reflection lines of polar angle $\alpha_{0}$ (with
$0\leq\alpha_{0}<180\text{°}$), or (ii) no reflection axis, in which
case, $i(\mathbf{q})=1/2$, i.e., the configuration contributes to
the sum in Eq. (\ref{eq:Morse1}) twice. A reflection line can be
(i) a ray passing through the vertex at $(0,0)$ and a point mass,
or (ii) the bisector of the angle with the vertex at $(0,0)$ and
the rays passing through two neighboring points. 
\end{enumerate}

\subsection{Uniqueness of the solutions}

To test if in a small neighborhood of each numerical solution there
is a unique exact solution, we use an approach based on the Krawczyk
operator method (Lee and Santoprete 2009; Moczurad and Zgliczynski
2019, 2020). This approach, that works hand in hand with the optimization
method used, is summarized below.

Consider an overdetermined system of nonlinear equations $\mathbf{f}(\mathbf{q})=\boldsymbol{0}$
with $\mathbf{q}\in\mathbb{R}^{N}$, $\mathbf{f}(\mathbf{q})=(f_{1}(\mathbf{q}),f_{2}(\mathbf{q}),\ldots,f_{M}(\mathbf{q}))^{T}\in\mathbb{R}^{M}$,
$M\geq N$, and the objective function $F(\mathbf{q})=(1/2)||\mathbf{f}(\mathbf{q})||^{2}$.
In the Gauss-Newton method for solving the least-squares problem $\min_{\mathbf{q}}F(\mathbf{q})$,
the search direction $\mathbf{p}_{k}=\mathbf{q}_{k+1}-\mathbf{q}_{k}$
satisfies the equation $\mathbf{J}_{f}^{T}(\mathbf{q}_{k})\mathbf{J}_{f}(\mathbf{q}_{k})\mathbf{p}=-\mathbf{J}_{f}^{T}(\mathbf{q}_{k})\mathbf{f}(\mathbf{q}_{k})$,
where $\mathbf{J}_{f}(\mathbf{q})=D\mathbf{f}(\mathbf{q})\in\mathbb{R}^{M\times N}$
is the Jacobian matrix of $\mathbf{f}$. In other words, the new iterate
is computed as
\begin{equation}
\mathbf{q}_{k+1}=\mathbf{q}_{k}-[\mathbf{J}_{f}^{T}(\mathbf{q}_{k})\mathbf{J}_{f}(\mathbf{q}_{k})]^{-1}\mathbf{J}_{f}^{T}(\mathbf{q}_{k})\mathbf{f}(\mathbf{q}_{k}).\label{eq:K1}
\end{equation}
An equivalent interpretation of the iteration formula (\ref{eq:K1})
can be given by taking into account that the gradient and the Gauss-Newton
approximation to the Hessian of $F(\mathbf{q})$ are given by $\mathbf{g}(\mathbf{q})=DF(\mathbf{q})=\mathbf{J}_{f}^{T}(\mathbf{q})\mathbf{f}(\mathbf{q})$
and $\mathbf{J}_{g}(\mathbf{q})=D\mathbf{g}(\mathbf{q})=D^{2}F(\mathbf{q})\approx\mathbf{J}_{f}^{T}(\mathbf{q})\mathbf{J}_{f}(\mathbf{q})$,
respectively. The first-order necessary condition for optimality requires
that $\mathbf{g}(\mathbf{q})=0$. If this equation is solved by means
of the Newton method we are led to the iteration formula 
\begin{equation}
\mathbf{q}_{k+1}=\mathbf{q}_{k}-[\mathbf{J}_{g}(\mathbf{q}_{k})]^{-1}\mathbf{g}(\mathbf{q}_{k}),\label{eq:K2}
\end{equation}
which coincides with that given by Eq. (\ref{eq:K1}). In order to
test the uniqueness, we use the interval arithmetic library INTLIB
(Kearfott et al. 1994), and let $[\mathbf{q}]_{r}\subset\mathbb{R}^{N}$
be an interval set centered at a numerical solution $\mathbf{q}$
with radius $r$ sufficiently small (e.g., $r=10^{-8}$). The procedure
involves two steps.
\begin{description}
\item [{Step}] \textbf{1}. Check if global minima of $F$ exist in $[\mathbf{q}]_{r}$,
that is, if $\mathbf{0}\in\mathbf{f}([\mathbf{q}]_{r})$.
\item [{Step}] \textbf{2}. Check if there exists a unique stationary point
of $F$ in $[\mathbf{q}]_{r}$, that is, if there exists a unique
zero of the first-order optimality equation $\mathbf{g}(\mathbf{q})=0$.
For this purpose, we define the Krawczyk operator by 
\[
K(\mathbf{q},[\mathbf{q}]_{r},\mathbf{g})=\mathbf{q}-\mathbf{C}\mathbf{g}(\mathbf{q})+[\mathbf{I}_{N\times N}-\mathbf{C}\mathbf{J}_{g}([\mathbf{q}]_{r})]([\mathbf{q}]_{r}-\mathbf{q}),
\]
where $\mathbf{C}\in\mathbb{R}^{N\times N}$ is a preconditioning
matrix which is sufficiently close to $[\mathbf{J}_{g}(\mathbf{q})]^{-1}$,
and $\mathbf{I}_{N\times N}$ is the identity matrix. A property of
the Krawczyk operator, which provide a method for proving the existence
of a unique zero in a given interval set, states that if $K(\mathbf{q},[\mathbf{q}]_{r},\mathbf{g})\subset\textrm{int}[\mathbf{q}]_{r}$
then there exists a unique zero of $\mathbf{g}$ in $[\mathbf{q}]_{r}$.
Following the recommendations of Kearfott (1996), the computational
process is organized as follows: 
\begin{enumerate}
\item compute the matrices $m(\mathbf{J}_{g}([\mathbf{q}]_{r})$ and $\Delta\mathbf{J}_{g}([\mathbf{q}]_{r})$
according to the decomposition $\mathbf{J}_{g}([\mathbf{q}]_{r})=m(\mathbf{J}_{g}([\mathbf{q}]_{r})+\Delta\mathbf{J}_{g}([\mathbf{q}]_{r})$,
where $m(\mathbf{J}_{g}([\mathbf{q}]_{r})$ is the center matrix of
the interval matrix $\mathbf{J}_{g}([\mathbf{q}]_{r})$ (the center
matrix of an interval matrix is defined componentwise according to
the rule $m([\underline{x},\overline{x}])=(\underline{x}+\overline{x})/2$); 
\item choose the preconditioning matrix $\mathbf{C}$ as the inverse of
$m(\mathbf{J}_{g}([\mathbf{q}]_{r})$, that is, $\mathbf{C}=[m(\mathbf{J}_{g}([\mathbf{q}]_{r})]^{-1}$;
\item to account for rounding errors in the calculation of $\mathbf{g}(\mathbf{q})$,
compute instead $\mathbf{g}([\mathbf{q}]_{r})$ by using interval
arithmetic; 
\item calculate $K(\mathbf{q},[\mathbf{q}]_{r},\mathbf{g})$ by means of
the relation
\[
K(\mathbf{q},[\mathbf{q}]_{r},\mathbf{g})=\mathbf{q}-\mathbf{C}\mathbf{g}([\mathbf{q}]_{r})-\mathbf{C}\,\Delta\mathbf{J}_{g}([\mathbf{q}]_{r})([\mathbf{q}]_{r}-\mathbf{q}).
\]
 
\end{enumerate}
\end{description}
Because any global minimum satisfies the first-order optimality condition,
we infer that if both conditions $\mathbf{0}\in\mathbf{f}([\mathbf{q}]_{r})$
and $K(\mathbf{q},[\mathbf{q}]_{r},\mathbf{g})\subset\textrm{int}[\mathbf{q}]_{r}$
are satisfied, then there exists a unique global minimum of $F$ in
$[\mathbf{q}]_{r}$, and so, a unique solution of the system of nonlinear
equations $\mathbf{f}(\mathbf{q})=\boldsymbol{0}$ in $[\mathbf{q}]_{r}$. 

For balanced configurations, the test is used with $M=N=2n$ and $\mathbf{f}(\mathbf{q})=(f_{1}(\mathbf{q}),f_{2}(\mathbf{q}),\ldots,f_{N}(\mathbf{q}))^{T},$
where $f_{2i-1}(\mathbf{q})$ and $f_{2i}(\mathbf{q})$, $i=1,\ldots,n$
are given by Eqs. (\ref{eq:Fx}) and (\ref{eq:Fy}), respectively.
For central configurations, we have to consider a system of nonlinear
equations, in which continuous rotations are eliminated. This is accomplished,
by rotating the configuration $\mathbf{q}$ such that the point mass
with the maximum radial distance is on the $x$-axis. Let $i_{0}$
be the index of this point mass and $\mathbf{q}^{\prime}$ the rotated
configuration. In addition to the functions $f_{2i-1}(\mathbf{q})$
and $f_{2i}(\mathbf{q})$, $i=1,\ldots,n$, we include the constraint
\begin{equation}
f_{2n+1}(\mathbf{q})=y_{i_{0}}\label{eq:AddConstraint}
\end{equation}
in the objective function $F(\mathbf{q})$, and apply the above test
in the case $M=2n+1>N=2n$ for the interval set $[\mathbf{q}^{\prime}]_{r}\subset\mathbb{R}^{N}$
centered at the rotated solution $\mathbf{q}^{\prime}$. 

Another option for checking the uniqueness is to perform a random
test. If the solution $\mathbf{q}\in Q$ is a unique global minimum
of $F$, that is, if (i) the gradient of $F$ vanishes at $\mathbf{q}$,
and (ii) the Hessian of $F$ is positive definite at $\mathbf{q}$,
then the quadratic approximation 
\begin{equation}
F(\mathbf{p}_{i})\approx\frac{1}{2}(\mathbf{p}_{i}-\mathbf{q})^{T}D^{2}F(\mathbf{q})(\mathbf{p}_{i}-\mathbf{q})\label{eq:Quad1}
\end{equation}
should be valid at a set of sample points $\{\mathbf{p}_{i}\}_{i=1}^{N_{\mathrm{q}}}\subset[\mathbf{q}]_{r}$
with $r$ sufficiently small (e.g., $r=10^{-3}||\mathbf{q}||$). As
before, we use the Gauss-Newton approximation to the Hessian matrix,
i.e., $D^{2}F(\mathbf{q})\approx\mathbf{J}_{f}^{T}(\mathbf{q})\mathbf{J}_{f}(\mathbf{q})$,
and generate $N_{\mathrm{q}}=10^{4}$ sample points $\mathbf{p}_{i}$
by using a pseudo-random number generator. For central configurations,
we either eliminate a sample $\mathbf{p}$ if it is a rotated version
of $\mathbf{q}$, or as before, include the constraint $f_{2n+1}(\mathbf{q})=y_{i_{0}}$
in the objective function $F(\mathbf{q})$. In practice, we test if
the RMS of the relative quadratic approximation errors
\begin{equation}
\textrm{RMS}(\mathbf{q})=\sqrt{\frac{\sum_{i=1}^{N_{\mathrm{q}}}\varepsilon(\mathbf{q},\mathbf{p}_{i})^{2}}{N_{\mathrm{q}}},}\label{eq:RMS}
\end{equation}
where (cf. Eq. (\ref{eq:Quad1})) 
\[
\varepsilon(\mathbf{q},\mathbf{p}_{i})=\frac{F(\mathbf{p}_{i})-\frac{1}{2}||\mathbf{J}_{f}(\mathbf{q})(\mathbf{p}_{i}-\mathbf{q})||^{2}}{F(\mathbf{p}_{i})}
\]
for $i=1,\ldots,N_{\mathrm{q}}$, is below a prescribed tolerance. 

\section{Numerical results}

The performances of the algorithm, and in particular, the number of
numerical solutions found, depend on the selection of a set of control
parameters. These are chosen as follows. 
\begin{enumerate}
\item The configuration $\mathbf{q}$ is considered to be an approximate
solution to the optimization problem, if $F(\mathbf{q})<10^{-20}$. 
\item The set of distinct solutions $Q$ contains only non-degenerate solutions.
If $\lambda_{i}$ are the eigenvalues of the Hessian matrix (\ref{eq:Hessmatrix})
sorted in ascending order, i.e., $|\lambda_{1}|\leq|\lambda_{2}|\leq\ldots\leq|\lambda_{2n}|$,
the central configuration $\mathbf{q}$ is assumed to be an approximate
degenerate solution if $|\lambda_{2}(\mathbf{q})|<10^{-15}$, while
for balanced configurations, the criterion is $|\lambda_{1}(\mathbf{q})|<10^{-15}$.
In this context, the Morse index of a central configuration $\mathbf{q}$
is given by the number of negative eigenvalues $\lambda_{i}$ of the
Hessian matrix for all $i\geq2$.
\item The ordered sets of mutual distances $R_{ij}=||\mathbf{q}_{i}-\mathbf{q}_{j}||$
and $R_{ij}^{\prime}=||\mathbf{q}_{i}^{\prime}-\mathbf{q}_{j}^{\prime}||$
corresponding to the configurations $\mathbf{q}=(\mathbf{q}_{1}^{T},...,\mathbf{q}_{n}^{T})^{T}$
and $\mathbf{q}^{\prime}=(\mathbf{q}_{1}^{\prime T},...,\mathbf{q}_{n}^{\prime T})^{T}$,
respectively, are considered to be approximately identical if $|R_{ij}-R_{ij}^{\prime}|\leq10^{-6}[1+\max(|R_{ij}|,|R_{ij}^{\prime})]$
for all $1\leq i<j\leq n$.
\end{enumerate}
The results of our numerical analysis are available at the website:
\emph{https://github.com}\\
\emph{/AlexandruDoicu/Balanced-and-Central-Configurations}. The output
file for each balanced configuration contains: (i) the value of the
objective function, (ii) the Cartesian coordinates of the point masses,
(iii) the residual of the normalization condition for the moment of
inertia, (iv) the Cartesian coordinates of the center of mass, (v)
the RMS of the relative quadratic approximation error and the maximum
quadratic approximation error, (vi) a logical flag indicating if global
minima of $F$ exist in a small box around the solution, as well as,
a logical flag indicating if there is a unique stationary point of
$F$ in the same box, and (vii) the number of degenerate solutions
and the corresponding Cartesian coordinates of the point masses. For
central configurations, the output file contains in addition (i) the
residual of the Albouy-Chenciner equations, (ii) the Morse and isotropy
indices, and (iii) the residual of the Morse equation. 

For $n\leq12$, the existence and local uniqueness of any central
configuration were independently tested by Moczurad and Zgliczynski
using their code based on the Krawczyk operator method (Moczurad and
Zgliczynski 2019).

\subsection{Central configurations}

The number of central configurations for $m=1$ and $\sigma_{x}=\sigma_{y}=1.0$
are listed in Table \ref{tab:NumberCC}. The results correspond to
a number of masses $n$ ranging from 3 to 12. The following comments
can be made here. 
\begin{enumerate}
\item In all cases, the Morse equality (\ref{eq:Morse1}) is satisfied,
the conditions $\mathbf{0}\in\mathbf{f}([\mathbf{q}]_{r})$ and $K(\mathbf{q},[\mathbf{q}]_{r},\mathbf{g})\subset\textrm{int}[\mathbf{q}]_{r}$
are fulfilled, and the RMS of the relative quadratic approximation
errors (\ref{eq:RMS}) is smaller than $2\times10^{-4}$.
\item For $n\leq10$, all tested sampling methods (Double-Box, Chaotic,
Faure, Sobol, Latin Hypercube, and Quasi-Oppositional Differential
Evolution) yield the same results. For $n=11$, only a Chaotic Method
leads to a set of central configurations for which the Morse equality
is satisfied, while for $n=12$ both a Chaotic Method and a Faure
sequence fulfill this desire. The case $n=11$, requiring large values
for $N_{\textrm{s}}^{0}$, $K$, and $k^{\star}$, is the most challenging. 
\item A Chaotic Method followed by a Faure sequence are the most efficient.
The reason is that the sets of starting points are much smaller than
the sets corresponding to other sampling methods. 
\item All central configurations corresponding to $n\leq9$ are identical
to those presented in Ferrario (2002). In this work, only 64 central
configurations that do not satisfy the Morse equality have been found
in the case $n=10$. The remaining configurations are illustrated
in Fig. \ref{fig:FerrarioFig}.
\item As compared to other solvers described in the literature (Lee and
Santoprete 2009; Moczurad and Zgliczynski 2019), the stochastic optimization
algorithm is extremely efficient. However, it should be pointed out
that the methods used in Moczurad and Zgliczynski (2019), and Lee
and Santoprete (2009) are rigorous, in the sense that the complete
list of central configurations is provided; in our approach, there
is no guarantee that the list is complete. 
\end{enumerate}
\begin{table}
\caption{The number of central configurations $N_{\textrm{sol}}$. Here, $n$
is the number of masses, $K$ the number of disjoint sample subsets,
$N_{\textrm{s}}^{0}$ the number of sample points in each subset,
$k^{\star}$ the number of sample subsets within $N_{\textrm{sol}}$
does not change, and $k_{0}$ the subset when $N_{\textrm{sol}}$
appears for the first time. The sampling methods are a Faure sequence
in the cases $n\protect\neq11$ and a Chaotic Method in the case $n=11$.
\label{tab:NumberCC}}

\medskip{}

\centering{}%
\begin{tabular}{ccccccc}
\hline 
$n$  & $N_{\textrm{sol}}$  & $N_{\textrm{s}}^{0}$  & $K$  & $k^{\star}$  & $k_{0}$  & $\begin{array}{c}
\textrm{Computational Time}\\
\textrm{hours:min:sec}
\end{array}$\tabularnewline
\hline 
3  & 2  & 1000  & 1000  & 100  & 1  & 0:00:01\tabularnewline
4  & 4  & 1000  & 1000  & 100  & 1  & 0:00:04\tabularnewline
5  & 5  & 1000  & 1000  & 100  & 1  & 0:00:10\tabularnewline
6  & 9  & 1000  & 1000  & 100  & 1  & 0:00:26\tabularnewline
7  & 14  & 1000  & 1000  & 100  & 2  & 0:01:05\tabularnewline
8  & 20  & 1000  & 1000  & 100  & 62  & 0:01:32\tabularnewline
9  & 42  & 1000  & 1000  & 100  & 34  & 0:02:04\tabularnewline
10  & 67  & 1000  & 1000  & 100  & 112  & 0:12:55\tabularnewline
11  & 114  & 2000  & 2000  & 500  & 534  & 0:32:24\tabularnewline
12  & 191  & 1000  & 1000  & 200  & 279  & 1:07:40\tabularnewline
\hline 
\end{tabular}
\end{table}

\begin{figure}
\includegraphics[scale=0.5]{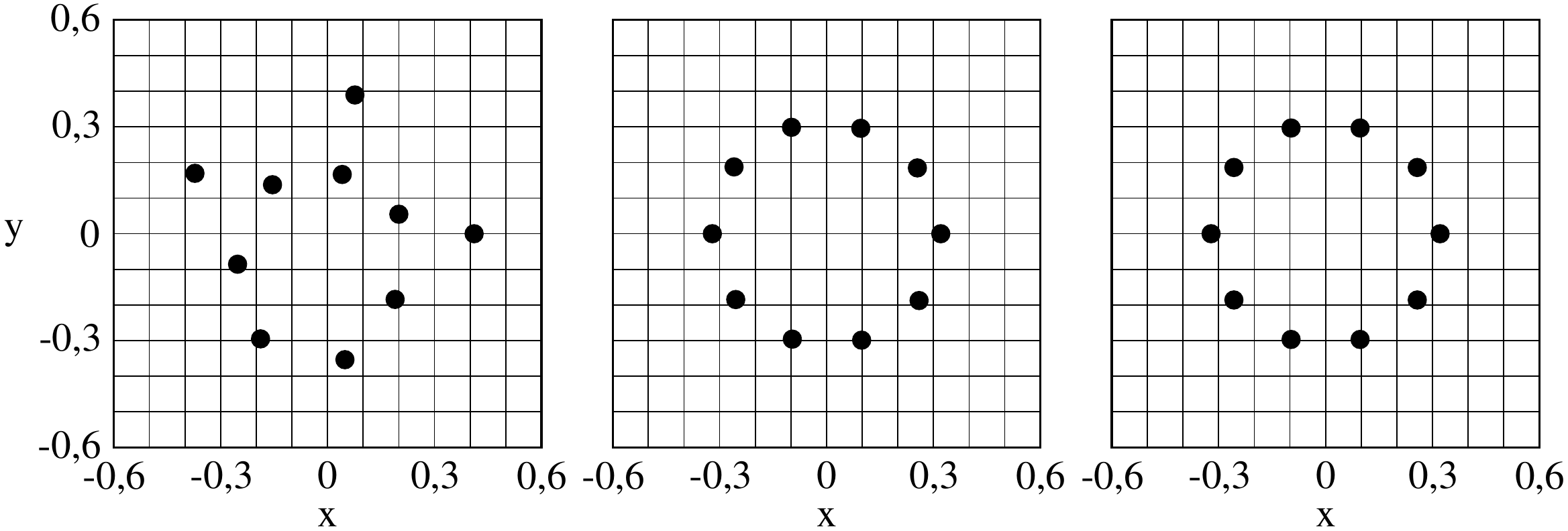}

\caption{The three central configurations that complete the list of Ferrario
(2002). Note that Configurations 2 and 3 do not coincide with the
regular decagon. Moreover, they are very similar but not identical;
Configuration 2 has (only) the axes of symmetry $\alpha_{1}=72\text{°}$
and $\alpha_{2}=162\text{°},$ while Configuration 3 has the axes
of symmetry $\alpha=0\text{°}$ and $\alpha_{2}=90\text{°}$. \label{fig:FerrarioFig}}
\end{figure}

\subsection{Balanced configurations\label{subsec:Balanced-configurations}}

The number of balanced configurations for five identical masses $m=1$
are listed in Table \ref{tab:NumberBC}, while the corresponding configurations
are illustrated in Figs. \ref{fig:BalanceConfigurations01}--\ref{fig:BC08}.
The sampling method was a Faure sequence. In these simulations, we
took $\sigma_{x}=1.0$ and changed $\sigma_{y}$ from $0.1$ to $0.8$
in steps of $0.1$. Thus, the ratio $\sigma_{y}/\sigma_{x}$ varies
from $0.1$ to $0.8$. For a comparison, the central configuration
($\sigma_{x}=\sigma_{y}=1.0)$ for five masses is shown in Fig. \ref{fig:Central-configurations}.
From our numerical analysis, the following conclusions can be drawn: 
\begin{enumerate}
\item In all cases, the collinear configurations along the $x$- and $y$-axis
are present. Moreover, the collinear configurations along the $x$-axis
are identical (as they should). 
\item Some configuration shapes appear in all test examples. This result
suggests that the balanced configurations can be classified according
to the similarity of their shapes. 
\item The shapes of the central configurations are also similar with, for
example, the shapes of the balanced configurations in the case $\sigma_{y}=0.8$
(the shapes (1), (2), (3), (4), and (5) are similar to the shapes
(10), (7), (2), (5), and (9), respectively). 
\end{enumerate}
\begin{table}
\caption{The number of balanced configurations $N_{\textrm{sol}}$ for $n=5$
masses. \label{tab:NumberBC}}

\medskip{}

\centering{}%
\begin{tabular}{ccccccc}
\hline 
$\sigma_{y}$  & $N_{\textrm{sol}}$  & $N_{\textrm{s}}^{0}$  & $K$  & $k^{\star}$  & $k_{0}$  & $\begin{array}{c}
\textrm{Computational Time}\\
\textrm{hours:min:sec}
\end{array}$\tabularnewline
\hline 
0.1  & 10  & 1000  & 1000  & 500  & 5  & 0:00:35\tabularnewline
0.2  & 11  & 1000  & 1000  & 500  & 5  & 0:00:57\tabularnewline
0.3  & 12  & 1000  & 1000  & 500  & 8  & 0:00:58\tabularnewline
0.4  & 15  & 1000  & 1000  & 500  & 20  & 0:01:34\tabularnewline
0.5  & 12  & 1000  & 1000  & 500  & 32  & 0:01:42\tabularnewline
0.6  & 12  & 1000  & 1000  & 500  & 32  & 0:02:04\tabularnewline
0.7  & 10  & 1000  & 1000  & 500  & 55  & 0:03:05\tabularnewline
0.8  & 10  & 1000  & 1000  & 500  & 62  & 0:05:05\tabularnewline
\hline 
\end{tabular}
\end{table}

For $n<8$, all central configurations have at least one axis of symmetry.
Although we did not perform an exhaustive numerical analysis, we make
some preliminary comments related to balanced configurations without
any axis of symmetry. 
\begin{enumerate}
\item In Chenciner (2017) it was asked whether in the case $n=4$, balanced
configurations without any axis of symmetry exist. Through a numerical
analysis, we found that for $\sigma_{x}=1.0$ and some specific values
of $\sigma_{y}$, such balanced configurations exist (Fig. \ref{fig:ASYM4}). 
\item In the case $n=10$, there are $67$ central configurations from which
$11$ are without any axis of symmetry. For the same number of point
masses, there are much more balanced configurations, and accordingly,
much more asymmetrical configurations. As an example, we mention that
for $n=10$, $\sigma_{x}=1.0$ and $\sigma_{y}=0.3$, we found $270$
balanced configurations from which $90$ are without any axis of symmetry.
Some examples are illustrated in Fig. \ref{fig:Asym10}. Note that
these results were computed by using a Chaotic Method and the control
parameters $N_{\textrm{s}}^{0}=K=2000$ and $k^{\star}=500$; the
computational time was $59$ minutes and $45$ seconds. 
\end{enumerate}
\begin{figure}
\includegraphics[scale=0.6]{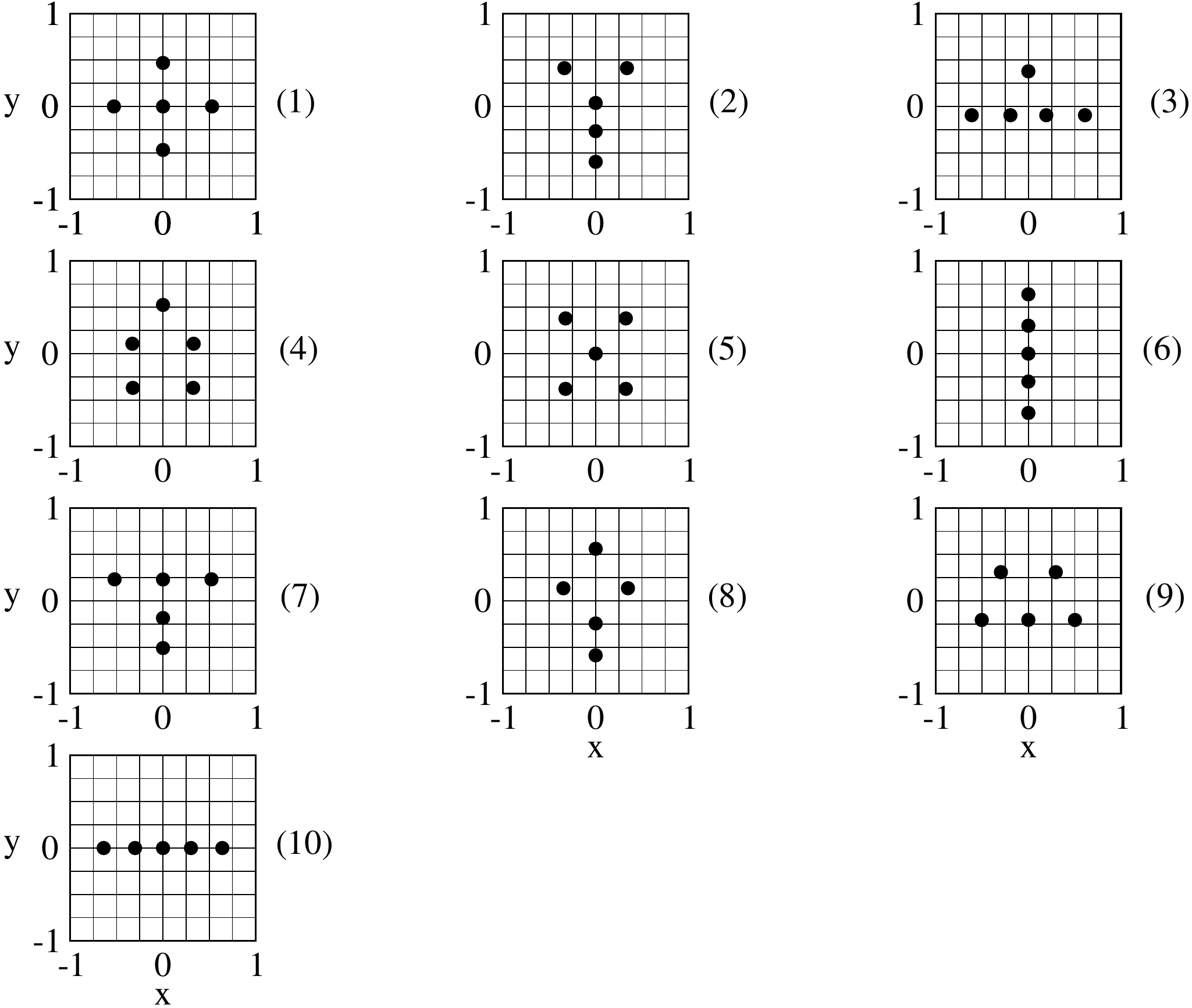}

\caption{Balanced configurations for five masses in the case $\sigma_{x}=1.0$
and $\sigma_{y}=0.1$. The $x_{i}$ and $y_{i}$ coordinates are normalized
according to the rules: $x_{i}\rightarrow\sqrt{m\sigma_{x}}x_{i}$
and $y_{i}\rightarrow\sqrt{m\sigma_{y}}y_{i}$, respectively. \label{fig:BalanceConfigurations01}}
\end{figure}

\begin{figure}
\includegraphics[scale=0.6]{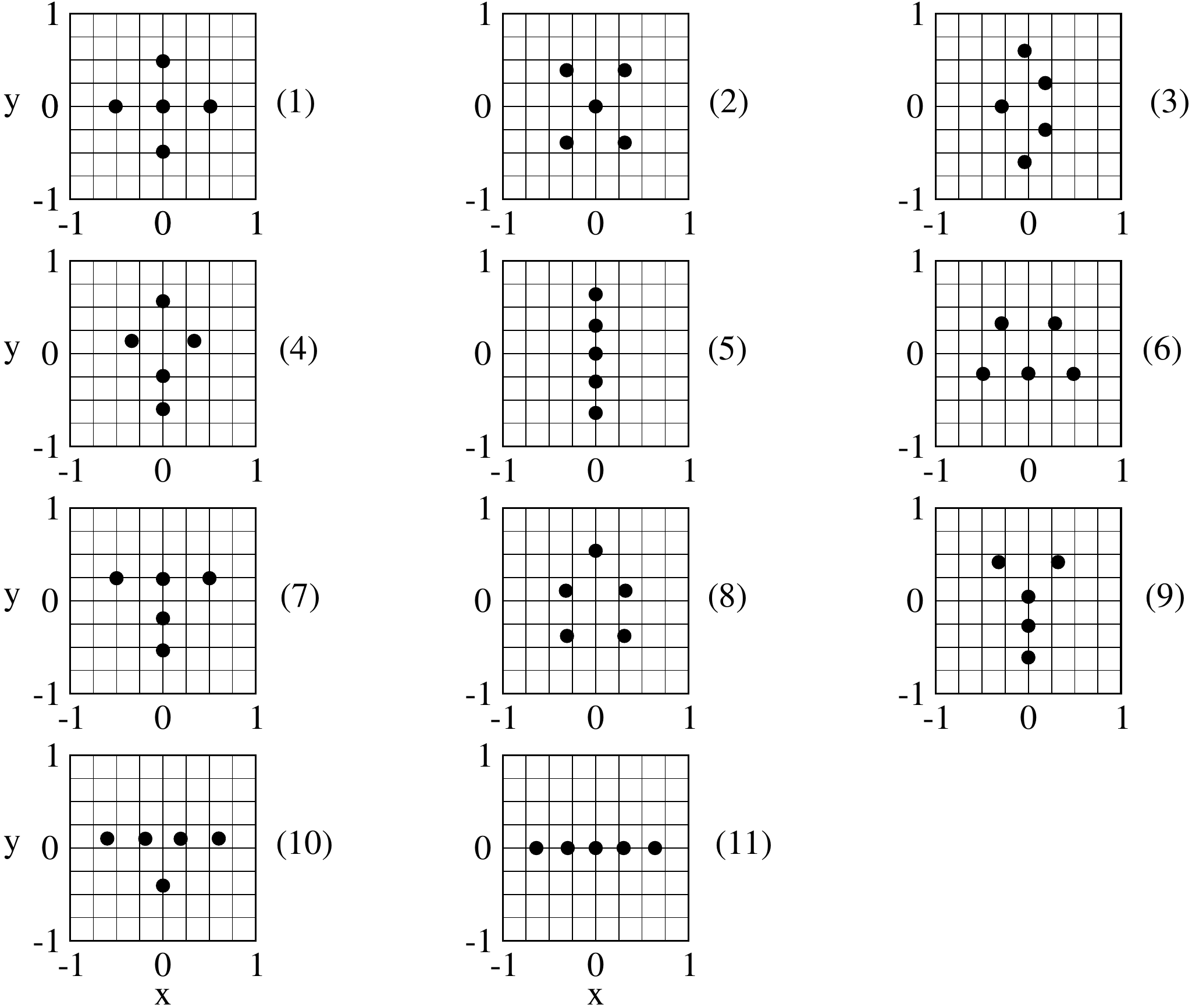}

\caption{The same as in Fig. \ref{fig:BalanceConfigurations01} but for $\sigma_{y}=0.2$.\label{fig:BC02}}
\end{figure}

\begin{figure}
\includegraphics[scale=0.6]{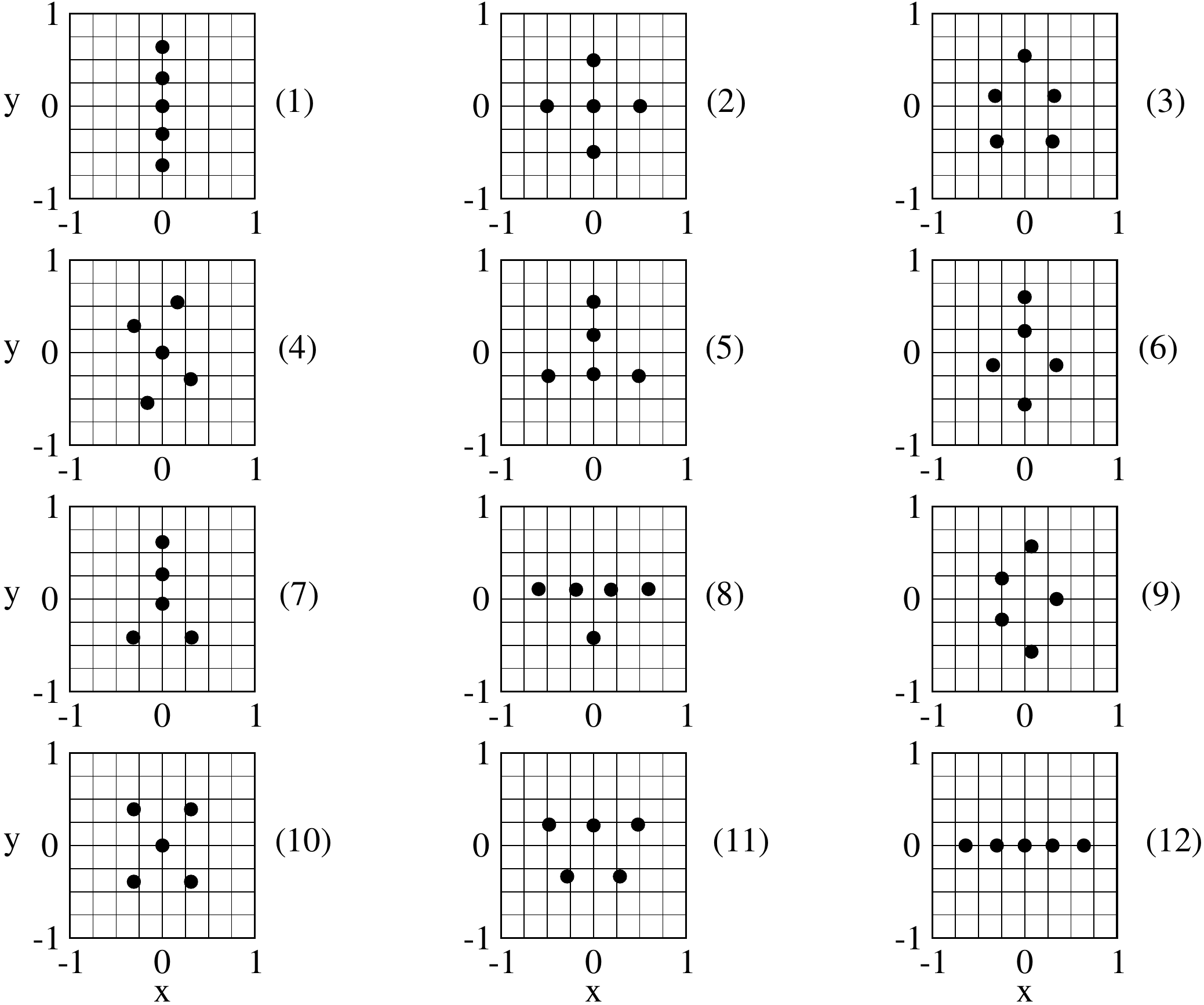}

\caption{The same as in Fig. \ref{fig:BalanceConfigurations01} but for $\sigma_{y}=0.3$.\label{fig:BC03}}
\end{figure}

\begin{figure}
\includegraphics[scale=0.6]{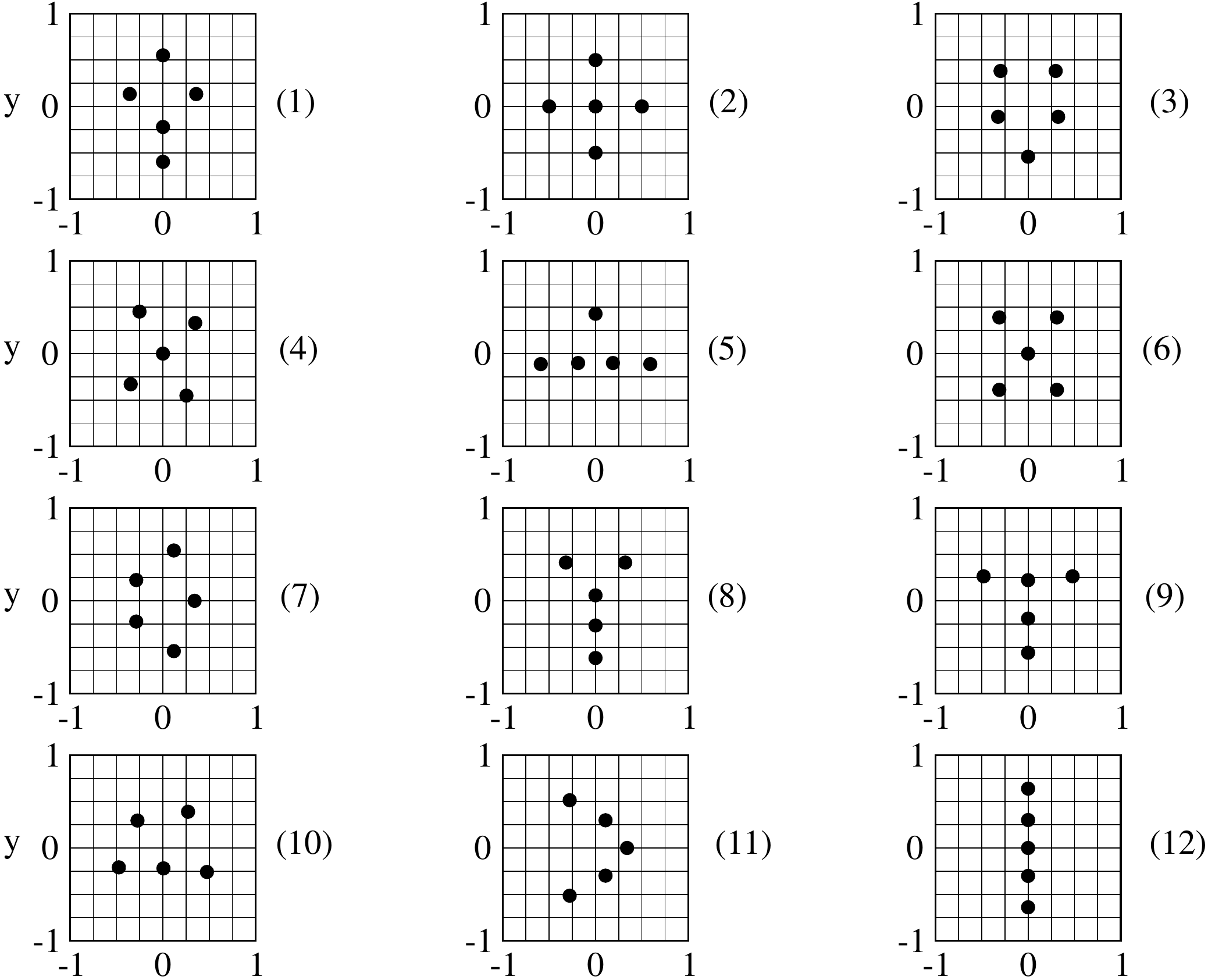}

\caption{The same as in Fig. \ref{fig:BalanceConfigurations01} but for $\sigma_{y}=0.4$.(continued
on next page)\label{fig:BC04}}
\end{figure}

\newpage{}

\begin{figure}
\includegraphics[scale=0.6]{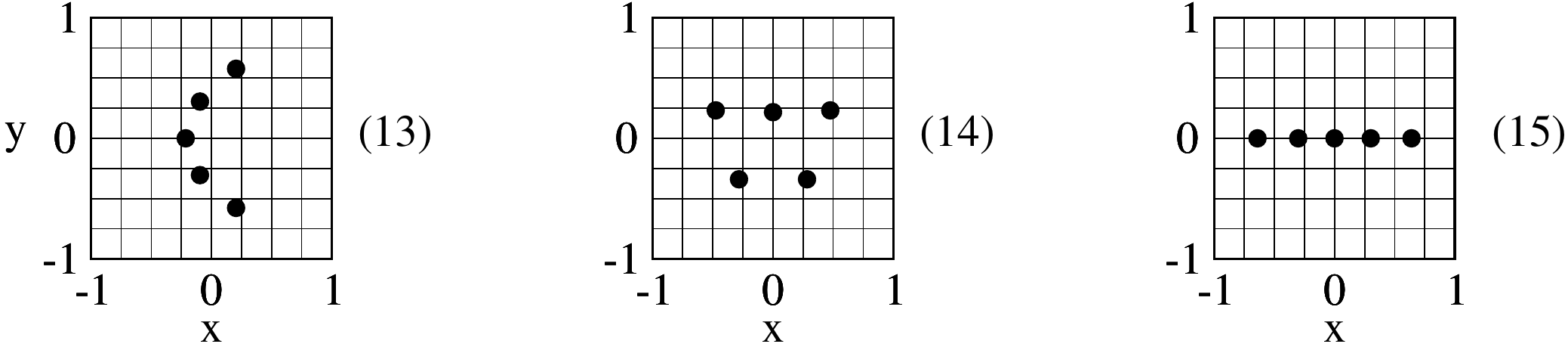}
\centering{}Continuation of Fig. \ref{fig:BC04}
\end{figure}

\section{Conclusions}

A stochastic optimization algorithm for analyzing planar central and
balanced configurations in the $n$-body problem has been developed.
The algorithm has been designed around the Minfinder method developed
by Tsoulos and Lagaris (2006) by including additional sampling and
local optimization methods. In the post-processing stage, several
solution tests have been incorporated. These are related to the fulfillment
of the normalization condition for the moment of inertia, the Albouy-Chenciner
equations, the Morse equality, and the uniqueness of the solutions.
Through a numerical analysis, we found an extensive list of central
configurations satisfying the Morse equality up to $n=12$. Although,
based on a random search, the algorithm is able to find the complete
list of central configurations (at least for $n\leq9$) with a low
computational time cost. For balanced configurations, we showed some
exemplary results in the case $n=5$, and some configurations without
any axis of symmetry in the cases $n=4$ and $n=10$.

The developed algorithm is versatile and has a wide range of application.
In addition to the Cartesian coordinates $(x_{i},y_{i})$, the masses
$m_{i}$, and the standard deviations $\sigma_{x}$ and $\sigma_{y}$
can be included in the inversion process. Actually, the unknown vector
$\mathbf{q}$ is defined as 
\[
\mathbf{q}=[(x_{1},y_{1}),\ldots,(x_{n},y_{n}),m_{1},\ldots,m_{n},\sigma_{x},\sigma_{y}]^{T},
\]
and a logical array specifies which components of the vector are considered
in the optimization process. Moreover, the algorithm can be directly
extended to spatial configurations and additional types of constraints,
similar to that given by Eq. (\ref{eq:AddConstraint}), can be taken
into account. In this way, various theoretical results, as for example,
the determination of central configurations for the $(n+1)$-body
problem or that of the so-called super central configurations (Xie
2010) for the $n$-body problem, can be first checked numerically. 

\pagebreak{}

\begin{figure}
\includegraphics[scale=0.6]{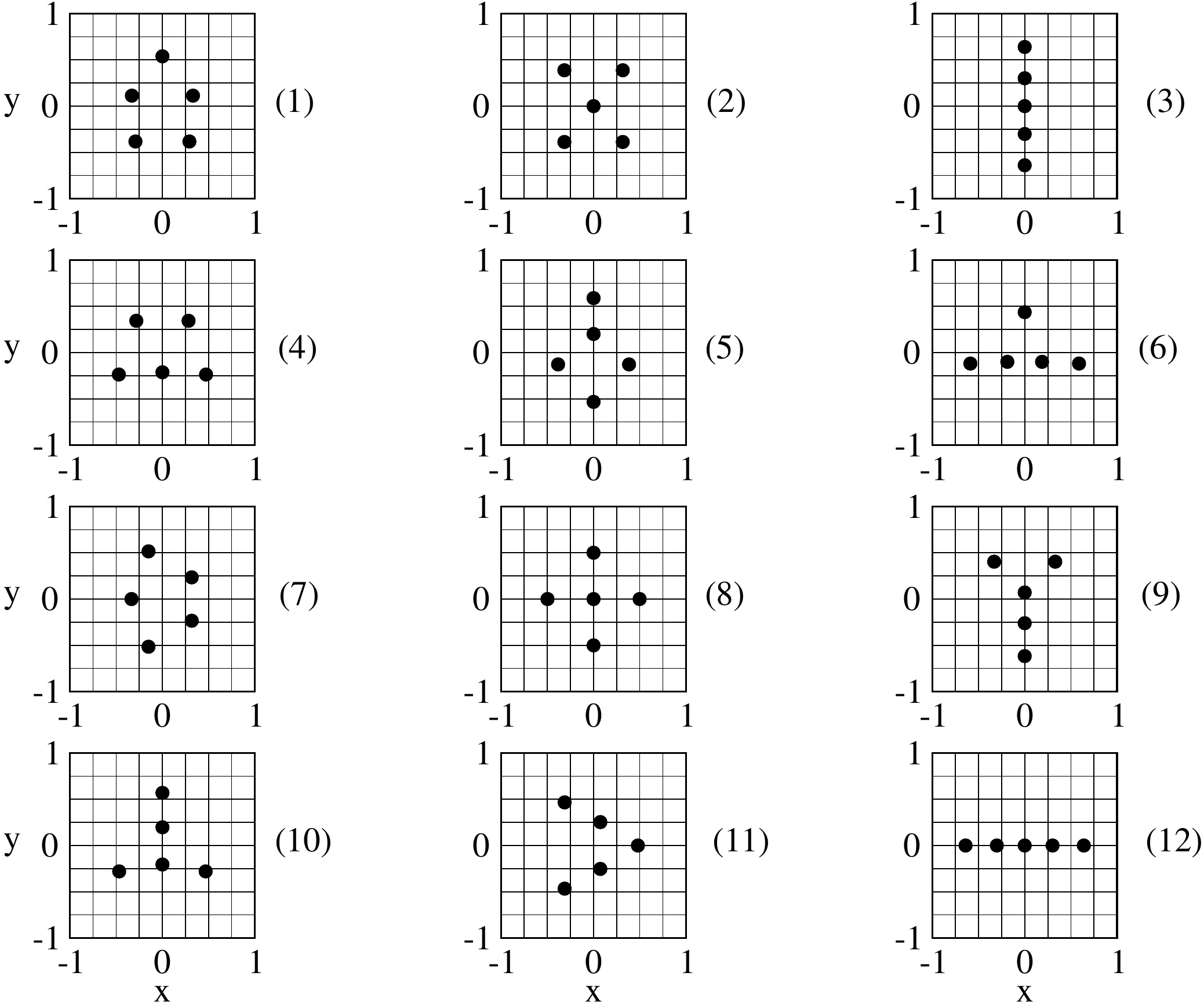}

\caption{The same as in Fig. \ref{fig:BalanceConfigurations01} but for $\sigma_{y}=0.5$.
\label{fig:BC05}}
\end{figure}

\pagebreak{}

\begin{figure}
\includegraphics[scale=0.6]{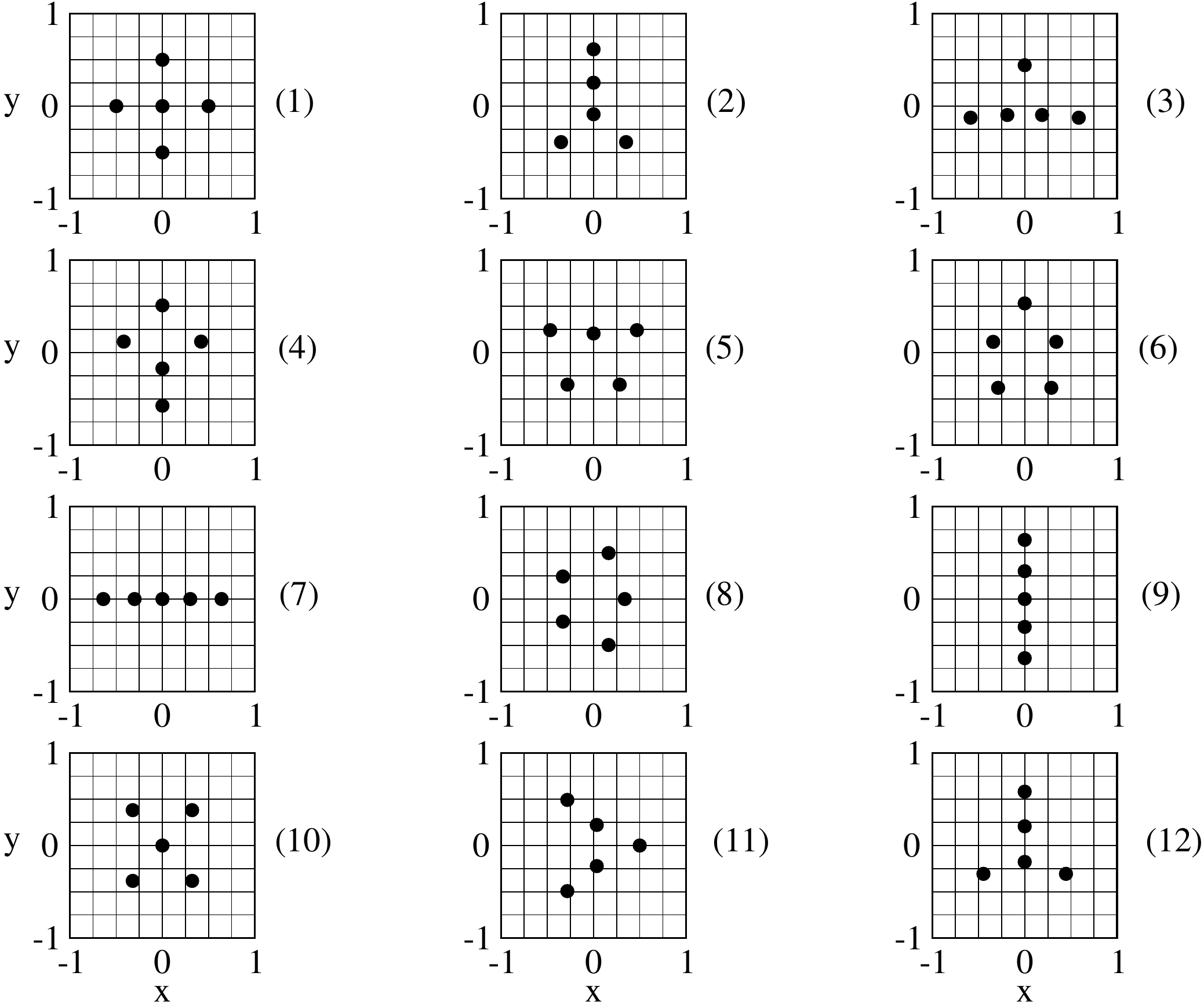}

\caption{The same as in Fig. \ref{fig:BalanceConfigurations01} but for $\sigma_{y}=0.6$.\label{fig:BC06} }
\end{figure}

\pagebreak{}

\begin{figure}
\includegraphics[scale=0.6]{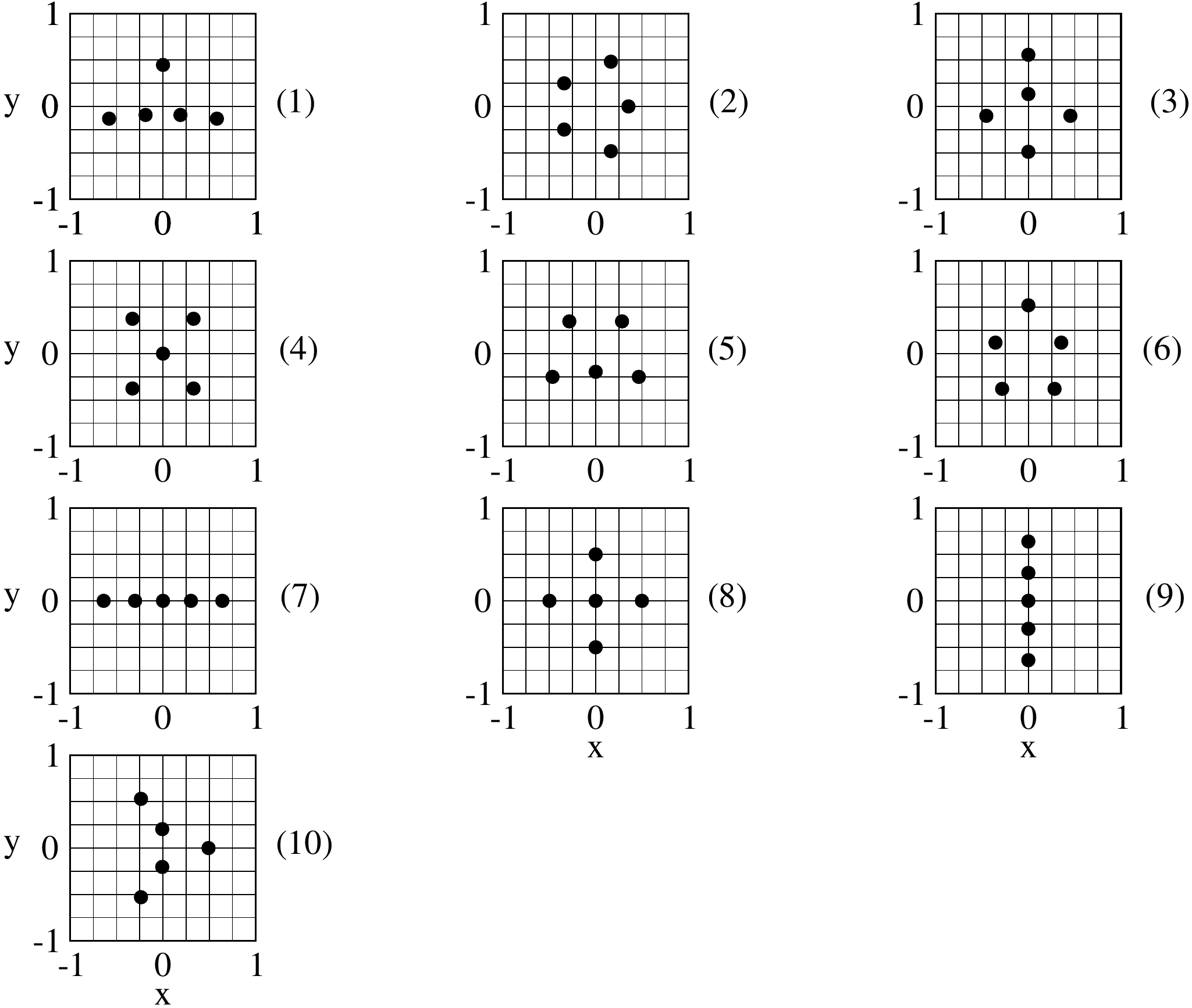}

\caption{The same as in Fig. \ref{fig:BalanceConfigurations01} but for $\sigma_{y}=0.7$.\label{fig:BC07}}
\end{figure}

\pagebreak{}

\begin{figure}
\includegraphics[scale=0.6]{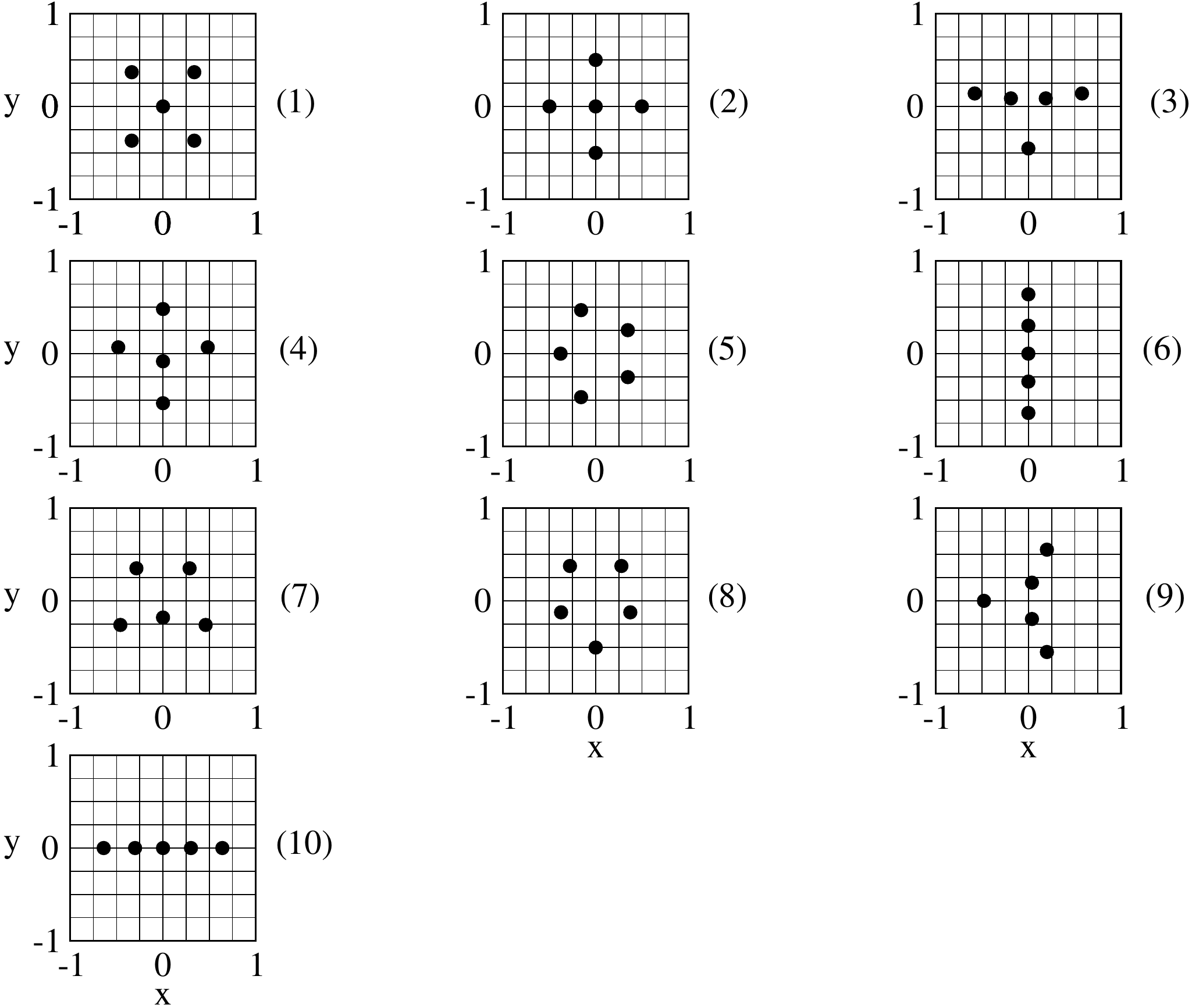}

\caption{The same as in Fig. \ref{fig:BalanceConfigurations01} but for $\sigma_{y}=0.8$.\label{fig:BC08}}
\end{figure}

\clearpage{}

\begin{figure}
\includegraphics[scale=0.6]{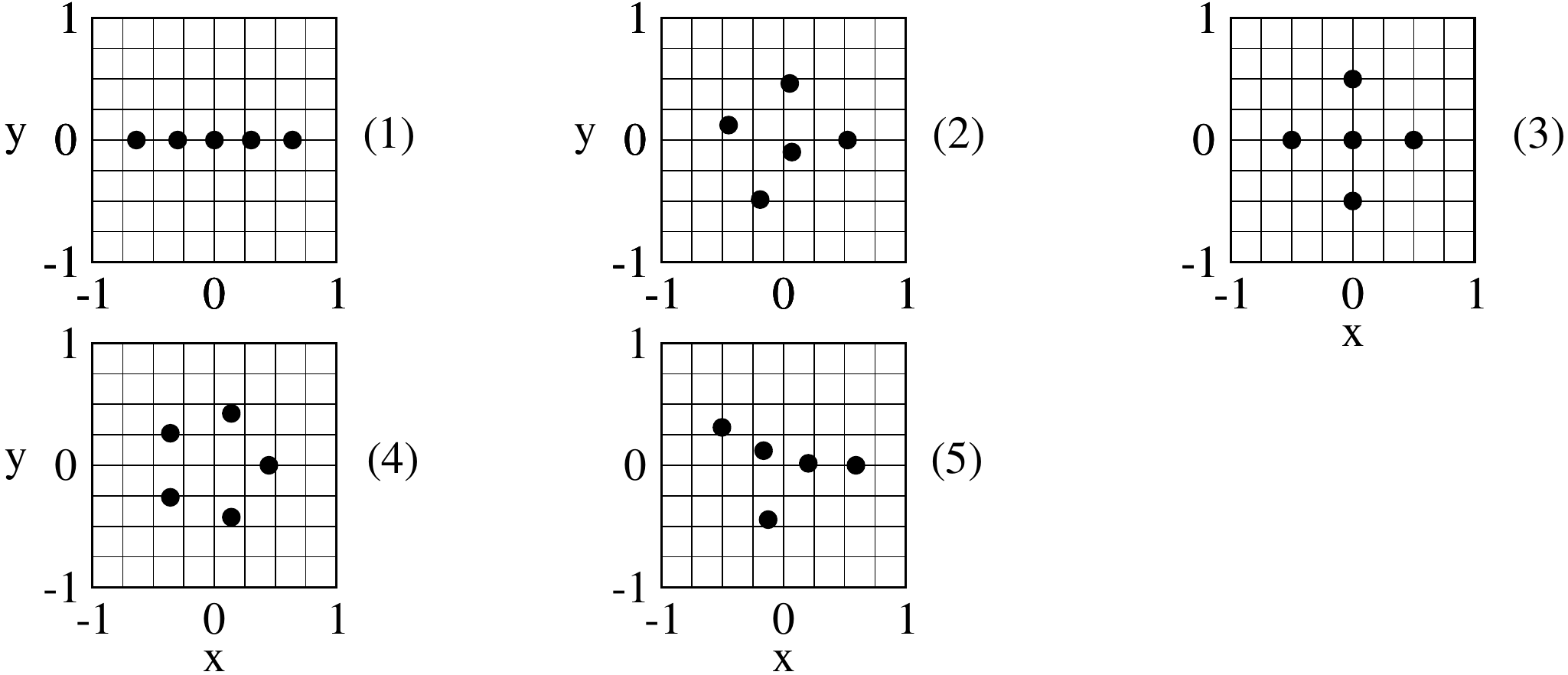}

\caption{Central configurations ($\sigma_{x}=\sigma_{y}=1.0$) for five masses.
The $x_{i}$ and $y_{i}$ coordinates are normalized as in Fig. \ref{fig:BalanceConfigurations01}.
\label{fig:Central-configurations}}
\end{figure}

\begin{figure}
\includegraphics[scale=0.5]{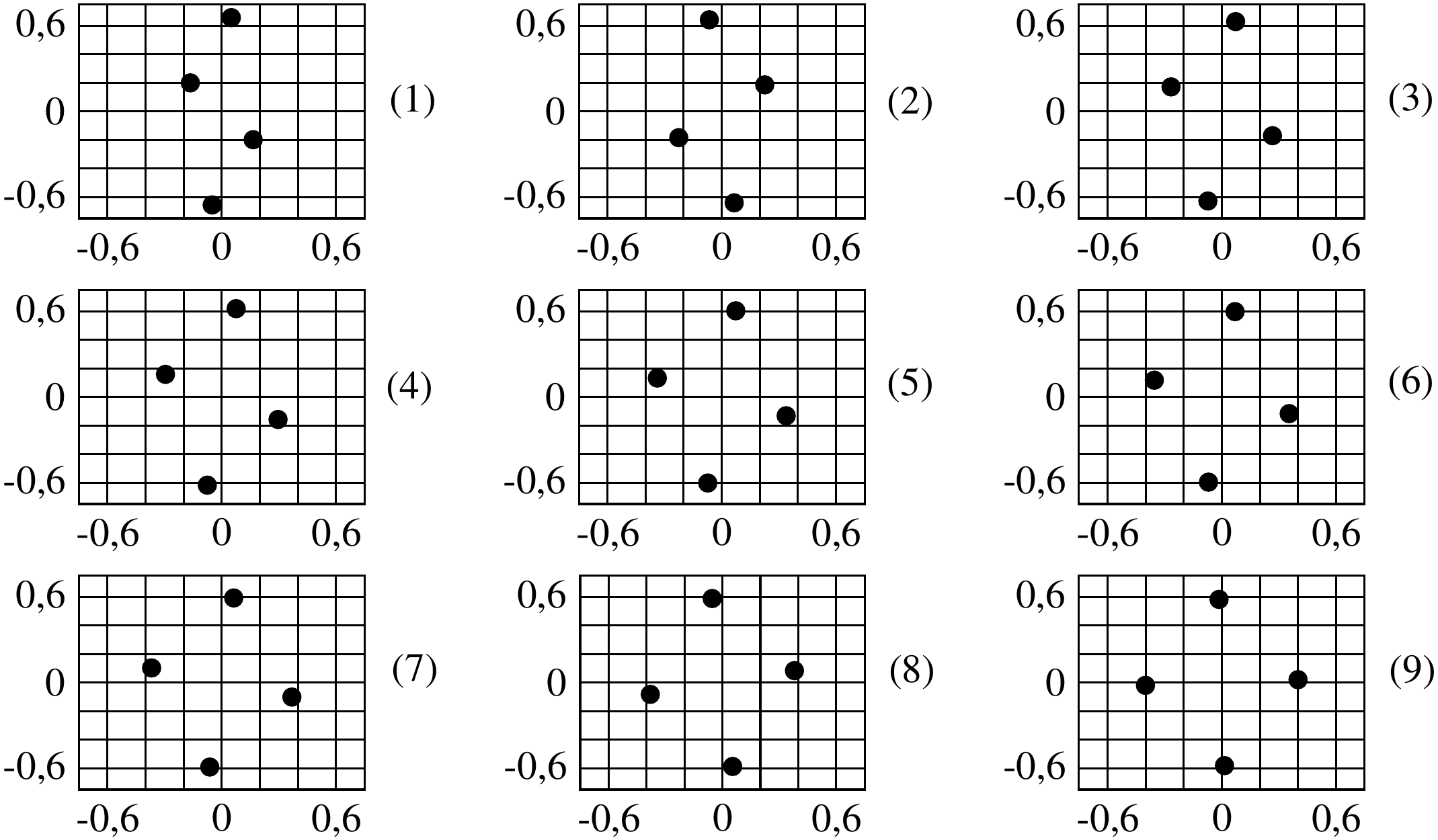}

\caption{Balanced configurations for four equal masses without any axis of
symmetry. They seem to be centrally symmetric. In these simulations,
$\sigma_{x}=1.0$ and $\sigma_{y}$ is increased from $0.25$ (panel
1) to $0.33$ (panel 9) in steps of $0.01$.\label{fig:ASYM4} }
\end{figure}

\clearpage{}

\begin{figure}
\includegraphics[scale=0.5]{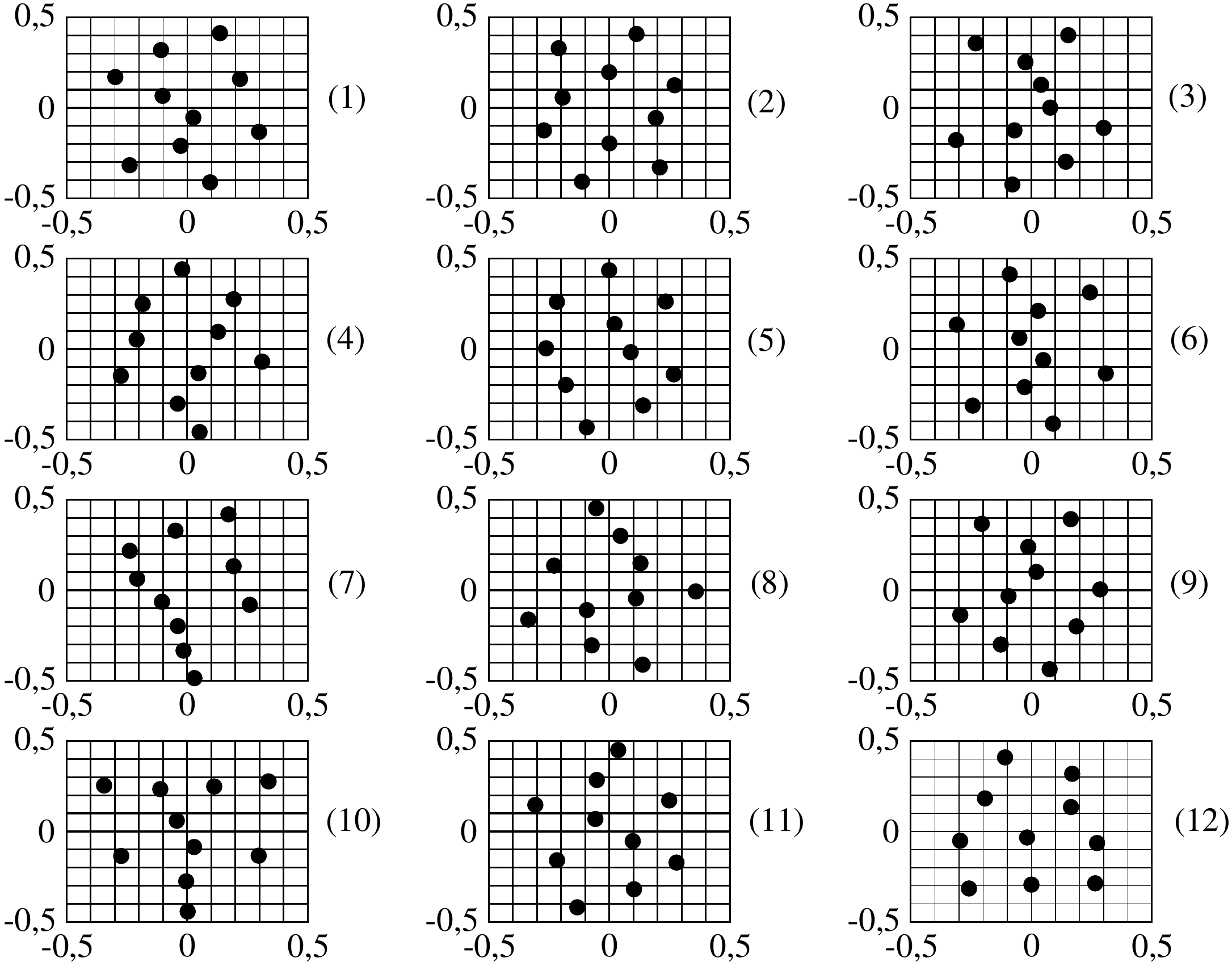}

\caption{Some balanced configurations without any symmetry for ten masses in
the case $\sigma_{x}=1.0$ and $\sigma_{y}=0.3$.\textcolor{blue}{{}
\label{fig:Asym10}}}
\end{figure}


\begin{acknowledgement*}
The authors thank Piotr Zgliczynski and Malgorzata Moczurad for their
helpful comments and critics, which bring various improvements of
this work, and for the independent verification of our solutions with
their code based on the Krawczyk operator method. Alexandru Doicu
and Lei Zhao were supported by DFG ZH 605/1-1.
\end{acknowledgement*}

\end{document}